\documentclass[12pt]{amsart}
\usepackage[headings]{fullpage}
\usepackage{amssymb,amsmath,amscd,bbm,tikz}
\usepackage[all,cmtip]{xy}
\usepackage{url}
\usepackage[bookmarks=true,%
    colorlinks=true,%
    linkcolor=blue,%
    citecolor=blue,%
    filecolor=blue,%
    menucolor=blue,%
    urlcolor=blue,%
    breaklinks=true]{hyperref}
\usepackage{slashed}    
\usepackage{verbatim}

\newtheorem{theorem}{Theorem}[section]
\theoremstyle{definition}
\newtheorem{proposition}[theorem]{Proposition}
\newtheorem{lemma}[theorem]{Lemma}

\newtheorem{remark}[theorem]{Remark}
\newtheorem{corollary}[theorem]{Corollary}

\def\BN{\mathbbm N}
\def\BZ{\mathbbm Z}
\def\BQ{\mathbbm Q}

\def\BC{\mathbbm C}

\def\BT{\mathbbm T}
\def\calA{\mathcal A}

\def\calC{\mathcal C}
\def\calT{\mathcal T}

\def\calS{\mathcal S}
\def\calB{\mathcal B}

\def\pt{\partial}
\def\ID{I_{\Delta}}

\def\a{\alpha}
\def\b{\beta}
\def\g{\gamma}

\def\ve{\varepsilon}
\def\th{\theta}

\def\Re{\mathrm{Re}}

\def\be{\begin{equation}}
\def\ee{\end{equation}}
\def\ID{I_{\Delta}}
\def\IKD{I^{\Delta}}
\def\Ipre{I^{\mathrm{pre}}}
\def\Ibal{I^{\mathrm{bal}}}
\def\Abar{\overline{A}}
\def\Bbar{\overline{B}}
\def\Cbar{\overline{C}}
\def\emu{e_{\mu}}
\def\elambda{e_{\lambda}}

\def\myu{\mathsf{u}}
\def\myv{\mathsf{v}}
\def\myfun{\eta}
\def\fourier{\mathsf{F}}
\def\mypos{\mathsf{q}}
\def\mymom{\mathsf{p}}
\def\mymu{\mu}
\def\myh{\mathsf{h}}
\def\myz{\mathsf{z}}
\def\myH{\mathsf{H}}
\def\myZ{\mathsf{Z}}
\def\myi{\mathsf{J}}
\newcommand{\im}{\mathsf{i}}

\newcommand{\poc}[2]{\left(#1\right)_{#2}}

\begin{document}
\title[A meromorphic extension of the 3D Index]{A meromorphic extension of 
the 3D Index}
\author{Stavros Garoufalidis}
\address{School of Mathematics \\
         Georgia Institute of Technology \\
         Atlanta, GA 30332-0160, USA \newline
         {\tt \url{http://www.math.gatech.edu/~stavros}}}
\email{stavros@math.gatech.edu}
\author{Rinat Kashaev}
\address{Section de Math\'ematiques, Universit\'e de Gen\`eve \\
2-4 rue du Li\`evre, Case Postale 64, 1211 Gen\`eve 4, Switzerland \newline
         {\tt \url{http://www.unige.ch/math/folks/kashaev}}}
\email{Rinat.Kashaev@unige.ch}

\date{June 24, 2017} 
\dedicatory{To Don Zagier, on the occasion of his 65th birthday}

\begin{abstract}
Using the locally compact abelian group $\BT \times \BZ$, we assign a 
meromorphic function to each ideal triangulation of a 3-manifold with torus
boundary components. The function is invariant under all 2--3
Pachner moves, and thus is a topological invariant of the underlying 
manifold. If the ideal triangulation has a strict angle structure, our
meromorphic function can be expanded into a Laurent power series whose 
coefficients are formal power series in $q$ with integer coefficients 
that coincide with the 3D index of \cite{DGG2}. Our meromorphic function 
can be computed explicitly from the matrix of the gluing equations of a 
triangulation, and we illustrate this with several examples.
\end{abstract}

\maketitle

{\footnotesize
\tableofcontents
}



\section{Introduction}
\label{sec.intro}

\subsection{The 3D-index of Dimofte-Gaiotto-Gukov}
\label{sub.intro}

In a recent breakthrough in mathematical physics, the physicists Dimofte, 
Gaiotto and Gukov \cite{DGG1,DGG2} introduced the \emph{3D-index}, a powerful 
new invariant of an ideal triangulation $\calT$ of a compact orientable 
3-manifold $M$ with non-empty boundary consisting of tori. 
The 3D-index was motivated by the study of the low energy limit of a famous 
6-dimensional $(2,0)$ superconformal field theory, and seems to contain 
a great deal of information about the geometry and topology of the 
ambient manifold. For suitable ideal triangulations, the 3D-index is 
a collection of formal Laurent power series in a variable $q$,
parametrized by a choice of peripheral homology class, i.e.,
an element of $H_1(\partial M,\BZ)$. 

Physics predicts that the 3D-index is independent of the triangulation 
$\calT$ and that it is a topological invariant of the ambient manifold.
However, there is a subtlety: the 3D-index itself (which is a sum over some
$q$-series over a lattice) is only defined for suitable triangulations, and 
it is invariant under 2--3 moves of such triangulations. It is not known
whether suitable triangulations are connected under 2--3 moves, and it is 
known that some 3-manifolds (for instance, the unknot) have no suitable 
triangulation. It was shown in~\cite{Ga:3D, GHRS} that a triangulation
is suitable if and only if it is 1-efficient, i.e.,
has no normal surfaces which are topologically 2-spheres or tori. 
Thus, the connected sum of two nontrivial knots, or the Whitehead double of 
a nontrivial knot has no 1-efficient triangulations.
With some additional work, one can extract from the 3D-index a topological 
invariant of hyperbolic 3-manifolds~\cite{GHRS}. 

This partial success in constructing a topological invariant suggests the
existence of an invariant of ideal triangulations unchanged under all 
2--3 Pachner moves. The construction of such an invariant is the goal
of our paper. Indeed, to any ideal triangulation, we associate
an invariant which is a meromorphic function of the peripheral variables, 
and, for triangulations with strict angle structures, the coefficients
of its expansion into Laurent series  coincides with the 3D-index of~\cite{DGG2}. 
Our meromorphic function is an example of a topological invariant associated 
to the self-dual locally compact abelian group (abbreviated LCA group) 
$\BT \times \BZ$. A more detailed formulation of our results follows.

In a sense, our paper does the opposite from that of~\cite{GK:qseries}. In
the latter paper we expressed state integral invariants (which are analytic
functions in a cut plane) in terms of $q$-series, whereas in the present paper
we assemble $q$-series into meromorphic state-integral invariants. Our work
illustrates the principle that some state-integrals can be formulated in 
terms of $q$-series and vice-versa. 

\subsection{Our results}
\label{sub.results}

Fix an ideal triangulation $\calT$ of an oriented 3-manifold $M$ whose boundary
consists of $r$ tori, and choose peripheral curves that form a basis of
$H_1(\pt M,\BZ)$. 
To simplify notation, we will present our results only in the case when 
$M$ has a single torus boundary though our statements and proofs
remain valid in the general case. 

After a choice of a meridian and
longitude, we can identify the complex torus $\BT_M=H^1(\pt M,\BC^*)$
with $(\BC^*)^2$ where the latter is given by the coordinates 
$(\emu,\elambda)$. Throughout the paper, $q$ will denote a 
complex number inside the unit disk: $|q|<1$. When $q=-e^h$ with $\Re(h) <0$
and $z \in \BC$, we define $(-q)^z=e^{z h}$. For
$r,s \in \BQ$, we define the associated $q$-rays of the complex torus by
\be
\label{eq.rse}
\Sigma_{r,s} = \{(\emu,\elambda) \in (\BC^*)^2 \,\, |
\,\, \emu^r \elambda^s \in q^\BN \}
\ee
where $\BN=\BZ_{\ge0}$ is the set of nonnegative integers. A shifted 
$q$-ray is a subset of the complex torus of the form $\ve q^t \Sigma_{r,s}$
for some $t \in \BQ$ and $\ve = \pm 1$.


\begin{theorem}
\label{thm.1}
With the above assumptions, there exists a meromorphic function
$$
I_{\calT}(q)\colon (\BC^*)^2\ni(\emu,\elambda) \mapsto 
I_{\calT,\emu,\elambda}(q)\in \BC\cup\{\infty\}
$$
with the following properties:
\begin{itemize}
\item[(a)]
$I_{\calT}(q)$ is invariant under 2--3 Pachner moves.
\item[(b)]
$I_{\calT}(q)$ is given by a balanced state-integral depending only on the 
Neumann--Zagier matrices of the gluing equations of $\calT$.
\item[(c)]
The singularities of $I_{\calT}(q)$ are contained the union of finitely many
shifted $q$-rays.
\end{itemize}
\end{theorem}

\begin{theorem}
\label{thm.2}
When $\calT$ has a strict angle structure, we have a Laurent series
expansion
\be
\label{eq.3d3d}
I_{\calT,\emu,\elambda}(q) = \sum_{(m, e) \in \BZ^2} \emu^m \elambda^e 
I_{\calT}(m,e)(q^2)
\ee
convergent on the unit torus $|\emu|=|\elambda|=1$, where $I_{\calT}(m,e)$
is the 3D-index of~\cite{DGG1}.
\end{theorem}

Fix a 3-manifold $M$ as above and consider the set $\calS_M$ of all ideal 
triangulations of $M$ that admit a strict angle structure. 

\begin{corollary}
\label{cor.constant}
Although it is not known yet if $\calS_M$ is connected or not by 2--3 Pachner moves, the
3D-index of~\cite{DGG1} is constant on $\calS_M$. 
\end{corollary}

\subsection{Discussion}
\label{sub.goal}

In a series of papers~\cite{AK:TQFT,AK:new,AK:complex,Kashaev:YB}, 
 topological invariants of 
(ideally triangulated) 3-manifolds have been constructed from certain self-dual LCA groups 
equipped with quantum dilogarithm functions. The main idea of those 
constructions is the following. Fixing 
a self-dual LCA group with a gaussian exponential and a quantum dilogarithm function, one 
assigns a state-integral invariant to an ideal 
triangulation decorated by a pre-angle structure (in the cited papers 
this is called \emph{shape structure}), that
is a choice of a strict angle structure within each ideal tetrahedron, but 
the angles do not have to add up to $2 \pi$ around the edges edges of the
ideal triangulation. The 
resulting state-integral is often the germ of a meromorphic function on the 
(affine vector) space of real-valued pre-angle structures.
This affine space has an (affine vector) subspace of complex-valued  
angle structures (the pre-angle structures that add up to $2\pi$ around each 
geometric edge of the triangulation).
The above meromorphic function is either infinity or restricts to a meromorphic function on the 
space of complex-valued angle structures. When the latter happens, the 
state-integral depends only on the peripheral angle monodromy. This way we 
obtain an invariant of ideal triangulations which depends on the peripheral
angle monodromy.

The above construction is general and, in particular, it applies to the invariants constructed in 
\cite{AK:TQFT,AK:complex,AK:new} and \cite{KLV}.
Our goal is to give a self-contained presentation in the case of the self-dual 
LCA
group $\BT \times \BZ$ with a quantum dilogarithm first found by Woronowicz 
in \cite{Woro:equalities} and to relate this invariant to the 3D-index of
Dimofte--Gaiotto--Gukov~\cite{DGG2}.


\section{Bulding blocks}
\label{sec.formulas}

\subsection{The tetrahedral weight}
\label{sub.tet.weight}

In this section we define the tetrahedral weight which is the building
block of our state-integral. We give a self-contained treatment
of the symmetries and identities that it satisfies.

Below, we will often consider expansions of meromorphic functions
defined on open annuli or punctured disks, examples of which
are given in Equations~\eqref{eq.psiID}, \eqref{eq.psi0}, \eqref{eq.psi0me}. 
These Laurent expansions (not to be confused with the formal Laurent
series which involve only finitely many negative powers and arbitrarily many
positive powers) are well-known in complex analysis and their existence, 
convergence and manipulation follows from Cauchy's theorem.  A detailed
discussion of this can be found, for example, in~\cite{Ahlfors}. 

As a warm up, recall the Pochhammer symbol 
\begin{equation}
\poc{x;q}{m}:=\prod_{i=0}^{m-1} (1-q^i x),\quad m\in\BN\cup\{\infty\},
\end{equation}
where $\BN:=\BZ_{\ge0}$ and we always assume that $|q|<1$. The next lemma 
summarizes the well-known properties of the 
Pochhammer symbol $\poc{x;q}{\infty}$.

\begin{lemma}
\label{lem.xqinfty}
The Pochhammer symbol $(x;q)_\infty$ has the following properties.

\noindent
\rm{(a)} It is an entire function of $x$ with simple zeros $x\in q^{-\BN}$.
\newline
\rm{(b)} It satisfies the $q$-difference equation
\be
\label{eq.xqinfty}
(x;q)_\infty =(1-x)(qx;q)_\infty \,.
\ee
\noindent
\rm{(c)} It has convergent power series expansions
\begin{subequations}
\begin{align}
\label{eq.poc1}
\frac{1}{\poc{x;q}{\infty}} &= \sum_{n=0}^\infty 
\frac{x^n}{\poc{q;q}{n}}, \quad |x|<1
\\
\label{eq.poc2}
\poc{x;q}{\infty} &= \sum_{n=0}^\infty 
(-1)^n \frac{q^{\frac{1}{2}n(n-1)} x^n}{\poc{q;q}{n}},\quad \forall x\in \BC.
\end{align}
\end{subequations}
 \end{lemma}
For the proof of part (c), see for instance \cite[Prop.~2]{Zagier:dilog}.

Consider the function 
\be
\label{eq.Gq}
G_q(z)=\frac{\poc{-qz^{-1};q}{\infty}}{\poc{z;q}{\infty}} =
\frac{\Bigl(1+\frac{q}{z}\Bigr)
\left(1+\frac{q^2}{z}\right)\left(1+\frac{q^3}{z}\right) \dots}{
(1-z)(1-qz)(1-q^2z) \dots} \,,
\ee
The next lemma summarizes its properties.

\begin{lemma}
\label{lem.Gq} 
The function $G_q(z)$ defined in \eqref{eq.Gq} has the following properties.

\noindent
\rm{(a)} It is a meromorphic function of $z \in \BC^*:=\BC\setminus\{0\}$ 
with simple zeros and poles in $-q^{1+\BN}$ and $q^{-\BN}$ respectively, and 
with essential singularities at $z=0$ and $z=\infty$.
\newline
\rm{(b)}
It satisfies the $q$-difference equation
\be
\label{eq.Gq23}
G_q(qz)=(1-z)(1+z^{-1}) G_q(z) 
\ee
and the involution equation
\be
\label{eq.Gq2}
G_q(-qz)=\frac{1}{G_q(z^{-1})} \,.
\ee
\rm{(c)}
It has a convergent Laurent series expansion in the punctured unit disk 
$0<|z|<1$:
\be
\label{eq.GJ}
G_q(z) = \sum_{n \in \BZ} J(n)(q) z^n 
\ee
where
\be
\label{eq.jdef}
J(n)(q):= \sum_{k=(-n)_+}^\infty 
\frac{q^{\frac{k(k+1)}{2}}}{(q)_k (q)_{n+k}} , \quad (n)_+:=\max\{n,0\},
\ee
is a well-defined element of $\BZ[[q]]$, analytic in the disk $|q|<1$.
\end{lemma}
Parts (a) and (b) follow easily from the product expansion of the Pochhammer 
symbol and part (c) follows from~\eqref{eq.poc1}-\eqref{eq.poc2}.

The \emph{tetrahedral weight} is a function $\psi^0(z,w)$ defined
by:
\be
\label{eq.psi0def}
\psi^0(z,w) := c(q) \, G_q(-qz) G_q(w^{-1})G_q(w z^{-1}) 
\ee
where
\be
\label{eq.cq}
c(q):= \frac{\poc{q;q}{\infty}^2}{\poc{q^2;q^2}{\infty}} \,.
\ee
The properties of this function are summarized in the following lemma.

\begin{lemma}
\label{lem.psi0} 
The function $\psi^0(z,w)$ defined in \eqref{eq.psi0def} has the following 
properties.

\noindent
\rm{(a)} It is a meromorphic function of 
$(z,w) \in (\BC^*)^2$ with zeros in
\be
\label{eq.zeros-z-w}
z \in q^{\BN}, \quad \text{or} \quad
w^{-1} \in -q^{1+\BN}, \quad \text{or} \quad z^{-1}w \in -q^{1+\BN} 
\ee
and poles in
\be
\label{eq.sing-z-w}
z \in -q^{-1-\BN}, \quad \text{or} \quad
w \in q^\BN, \quad \text{or} \quad zw^{-1} \in q^\BN \,.
\ee
\rm{(b)} It satisfies the $q$-difference equations
\begin{subequations}
\begin{align}
\label{eq.psi01}
w^{-1}\psi^0(qz,w) -\psi^0(z,w) -q^{-1}z^{-1} \psi^0(z,w) &=0, \\
\label{eq.psi02}
w \psi^0(qz,w) -\psi^0(z,w) -q z \psi^0(z,w) &=0 \,.
\end{align}
\end{subequations}
\rm{(c)} It satisfies the $\BZ/2$ and $\BZ/3$-invariance equations
\begin{subequations}
\begin{align}
\label{eq.psi022}
\psi^0(z,w) &= \psi^0(-q^{-1} w^{-1}, -q^{-1} z^{-1}) 
\\
\label{eq.psi03}
\psi^0(z,w) &= \psi^0(-q^{-1} z^{-1} w, -q^{-1} z^{-1})
= \psi^0(-q^{-1} w^{-1}, z w^{-1}) \,.
\end{align}
\end{subequations}
\rm{(d)} 
In the domain
\begin{equation}
\label{eq.psi0R}
1 < |w| < |z| < |q|^{-1},
\end{equation}
we have the absolutely convergent expansion
\be
\label{eq.psi0J}
\psi^0(z,w) = c(q) 
\sum_{e,m  \in \BZ} z^e w^m 
\sum_{\substack{k_1,k_2,k_3\in\BZ\\ k_1-k_3=e;\ k_3-k_2=m}} (-q)^{k_1} J(k_1)(q) 
J(k_2)(q) J(k_3)(q) 
\ee
where the interior sum is a well-defined element of $\BZ[[q]]$, analytic in
the disk $|q|<1$.
\end{lemma}
These properties follow from the definition of $\psi^0$ and the
properties of $G_q$ listed in Lemma~\ref{lem.Gq}.

\subsection{The quantum dilogarithm}
\label{sub.QDL}

In this sub-section we identify the tetrahedral weight function $\psi^0(z,w)$ 
with (the reciprocal of) the quantum dilogarithm function $\psi(z,m)$ 
on the self-dual LCA group $\BT \times \BZ$, given 
by~\cite[Eqn.97]{Kashaev:YB} 
\begin{align}
\label{eq.psi}
\psi(z,m) &= \frac{(-q^{1-m}/z;q^2)_\infty}{(-q^{1-m} z;q^2)_\infty} \,.
\end{align}
In the context of quantum groups, this function first appeared 
in~\cite{Woro:equalities}. In Appendix~\ref{sec.found}, we explain how 
formula~\eqref{eq.psi} fits the general definition of a quantum dilogarithm
over the LCA group $\BT \times \BZ$.

Lemma~\ref{lem.xqinfty} and the above definition imply the following
properties of the function $\psi(z,w)$.

\begin{lemma}
\label{lem.psizw} 
The function $\psi(z,m)$ defined in \eqref{eq.psi} has the following 
properties.

\noindent
\rm{(a)} It is a meromorphic function of $z$ with 
simple poles and zeros at $z\in-q^{-1-|m|-2\BN}$ and $z\in-q^{1+|m|+2\BN}$, 
respectively. 
\newline
\rm{(b)}
It is analytic in the annulus $0 < |z| < |q|^{-1-|m|}$ 
(which always includes the unit circle $|z|=1$) where it has an absolutely 
convergent Laurent series expansion
\be
\label{eq.psiID}
\psi(z,m) = \sum_{e \in \BZ} \IKD(m,e)(q) z^e
\ee
where 
\be
\label{eq.3d2}
\IKD(m,e)(q) = (-q)^e \ID(m,e)(q^2)
\ee
is related to the tetrahedron index $\ID$ of~\cite{DGG2} given by
\be
\label{eq.ID}
\ID(m,e)(q)=\sum_{n=(-e)_+}^\infty (-1)^n \frac{q^{\frac{1}{2}n(n+1)
-\left(n+\frac{1}{2}e\right)m}}{(q)_n(q)_{n+e}} \in \BZ[[q^{\frac{1}{2}}]]
\ee
and $(e)_+:=\max\{0,e\}$ and $(q)_n:=\poc{q;q}{n}=\prod_{i=1}^n (1-q^i)$. 
\end{lemma}

The next theorem connects the tetrahedral weight function $\psi^0(z,w)$ with 
the above function $\psi(z,m)$.

\begin{theorem}
\label{thm.psi0}
\rm{(a)} In the domain~\eqref{eq.psi0R} we have the identity
\begin{align}
\label{eq.psi0}
\psi^0(z,w) &= \sum_{m \in \BZ} \psi(z,m) w^m 
\end{align}
where the sum is absolutely convergent.

\rm{(b)} In the domain~\eqref{eq.psi0R} we have an absolutely convergent 
double Laurent series expansion
\be
\label{eq.psi0me}
\psi^0(z,w) = \sum_{e,m  \in \BZ} \IKD(m,e)(q) z^e w^m \,.
\ee
\end{theorem}

\begin{proof}
We let RHS denote the sum in the right hand side of Equation~\eqref{eq.psi0}.
RHS is absolutely convergent in the domain~\eqref{eq.psi0R}
and it can be explicitly calculated by using 
Ramanujan's $_1\psi_1$-summation formula. The detailed computation
appears in~\cite[Eqn.~(98)]{Kashaev:YB} and the result reads
\be
\label{eq.psi^0}
\mathrm{RHS} =\frac{\poc{q^2;q^2}{\infty}\poc{z^{-2};q^2}{\infty}}{
\poc{w^{-2};q^2}{\infty}\poc{w^2z^{-2};q^2}{\infty}}
\left(\frac{\theta_q(zw^{-2})}{
\theta_q(z)}+w\frac{\theta_q(qw^2/z)}{\theta_q(q/z)}\right)
\ee
where 
\be
\label{eq.theta_q}
\theta_q(x):=\sum_{k\in\BZ}q^{k^2} x^k=
\poc{q^2;q^2}{\infty}\poc{-qx;q^2}{\infty}\poc{-q/x;q^2}{\infty}
\ee
is the Jacobi theta function, and the second equality in~\eqref{eq.theta_q} 
is the Jacobi triple product identity.
By using Lemma~\ref{lem.dopsum} (see below) and the Jacobi triple
product identity 
we can further simplify the right hand side of~\eqref{eq.psi^0}, thus
getting 

\begin{align*}
\mathrm{RHS} &=\frac{\poc{q;q}{\infty}^2\poc{z^{-1};q}{\infty}
\poc{-qw;q}{\infty}\poc{-qz/w;q}{\infty}}{\poc{q^2;q^2}{\infty}
\poc{-qz;q}{\infty}\poc{w^{-1};q}{\infty}\poc{w/z;q}{\infty}} \\
&= c(q) \, G_q(-qz) G_q(w^{-1})G_q(w z^{-1}) \\
&= \psi^0(z,w) \,.
\end{align*}
This concludes the proof of the part (a).
Part (b) follows from \eqref{eq.psi0} combined with \eqref{eq.psiID}.
\end{proof}

\begin{lemma}
\label{lem.dopsum}
For any choice of the square root $p:=\sqrt{q}$, we have the identity 
\be
\label{eq:dopsum}
\frac{\theta_q(zw^{-2})
}{\theta_q(z)}+w\frac{\theta_q(qw^2/z)}{\theta_q(q/z)}=
\frac{\poc{q;q}{\infty}}{\poc{q^2;q^2}{\infty}^2}\frac{\theta_{p}
\left(pw/z\right)\theta_{p}\left(w/p\right)}{\theta_{p}\left(z/p\right)}.
\ee
\end{lemma}

\begin{proof}
Denoting the left hand side of \eqref{eq:dopsum} as $f(z^{-1},w)$, we have
$$
f(u,w)=h(u,w)/g(u) \,,
$$
where
$$
g(u):=\theta_q(u)\theta_q(qu),
\qquad h(u,w):=\theta_q(qu)\theta_q(uw^2)+w\theta_q(u)\theta_q(quw^2) \,.
$$
First, we rewrite $g(u)$ as follows:
\begin{multline}
g(u)=\poc{q^2;q^2}{\infty}^2\poc{-qu;q^2}{\infty}
\poc{-q^2u;q^2}{\infty}\poc{-q/u;q^2}{\infty}\poc{-1/u;q^2}{\infty}\\
=\poc{q^2;q^2}{\infty}^2\poc{-qu;q}{\infty}\poc{-1/u;q}{\infty}
=\frac{\poc{q^2;q^2}{\infty}^2}{\poc{q;q}{\infty}}\theta_p(pu).
\end{multline}
Then, we transform $h(u/w,w)$ as follows:
\begin{equation}
h\left(\frac uw,w\right)=\sum_{k,l\in\BZ}q^{k^2+l^2}u^{k+l}w^{l-k}(q^k+q^l w)
=\sum_{k,l\in\BZ}q^{k^2-2kl+2l^2}u^{k}w^{2l-k}(q^{k-l}+q^l w)
\end{equation}
where in the second equality we have shifted the summation variable 
$k\mapsto k-l$. Next, we write out separately the sum over even and odd $k$: 
\begin{multline}
h\left(\frac uw,w\right)=\sum_{k,l\in\BZ}q^{2k^2+2(l-k)^2}u^{2k}w^{2l-2k}
\left(q^{2k-l}+q^l w+q^{4k+1-2l}uw^{-1}(q^{2k+1-l}+q^l w)\right)\\
=\sum_{k,l\in\BZ}q^{2k^2+2l^2+k}u^{2k}w^{2l}
\left(q^{-l}+q^{l} w+q^{2k+1-2l}uw^{-1}(q^{1-l}+q^{l} w)\right)
\end{multline}
where, this time, in the second equality we have shifted the summation 
variable $l\mapsto l+k$. Now, we can absorb both summations by using the 
definition of the $\theta$-function:
\begin{multline}
h\left(\frac uw,w\right)\\=\theta_{q^2}(qu^2)\left(\theta_{q^2}(q^{-1}w^2)
+w\theta_{q^2}(qw^2)\right)
+qu\theta_{q^2}(q^3u^2)\left(qw^{-1} \theta_{q^2}(q^{-3}w^2) 
+\theta_{q^2}(q^{-1}w^2)\right).
\end{multline}
Finally, by using the functional equation
\begin{equation}\label{eq:fe-theta}
\theta_q(q^{2l}x)=q^{-l^2}x^{-l}\theta_q(x),\qquad \text{for all} \qquad l\in\BZ,
\end{equation}
with $q\mapsto q^2$, $x=w^2$ and $l=-1$ in the second term, we arrive to the 
following factorized formula
\begin{equation}
h\left(\frac uw,w\right)
=\left(\theta_{q^2}(qu^2)+qu\theta_{q^2}(q^3u^2)\right)
\left(\theta_{q^2}(q^{-1}w^2)+w\theta_{q^2}(qw^2)\right)=s(qu)s(w)
\end{equation}
where
\begin{multline}
s(w):=\theta_{q^2}(q^{-1}w^2)+w\theta_{q^2}(qw^2)
=\sum_{k\in\BZ}q^{2k^2}w^{2k}\left(q^{-k}+q^{k}w\right)\\
=\sum_{k\in\BZ}p^{4k^2}w^{2k}\left(p^{-2k}+p^{2k}w\right)
=\sum_{k\in\BZ}\left(p^{(2k)^2-2k}w^{2k}+p^{(2k+1)^2-2k-1}w^{2k+1}\right)\\
=\sum_{k\in\BZ}p^{k^2-k}w^{k}=\theta_p(w/p).
\end{multline}
Thus, we have obtained the equality
\begin{equation}
h(u,w)=\theta_p(puw)\theta_p(w/p),
\end{equation}
and formula~\eqref{eq:dopsum} follows straightforwardly.
\end{proof}
In summary, the tetrahedral weight $\psi^0(z,w)$ is given by two sum formulas
and a product formula, and the last of which implies the meromorphicity of the
function and the location of its zeros and poles:
\begin{align}
\psi^0(z,w)&= \sum_{m \in \BZ} \psi(z,m) w^m \\
&= \sum_{e,m  \in \BZ} \IKD(m,e)(q) z^e w^m \\
&= c(q) \, G_q(-qz) G_q(w^{-1})G_q(w z^{-1}) \,.
\end{align}

For completeness, the next lemma summarizes the $q$-difference equations
and the symmetries of $\psi(z,w)$.

\begin{lemma}
\label{lem.psi} The function $\psi(z,m)$ defined in \eqref{eq.psi} 
satisfies the following $q$-difference equations
\begin{subequations}
\begin{align}
\label{eq.psi1}
\psi(qz,m+1) -\psi(z,m) -q^{-m-1}z^{-1} \psi(z,m) &= 0 \\
\label{eq.psi2}
\psi(qz,m-1) -\psi(z,m) -q^{-m+1}z \psi(z,m) &= 0 \,.
\end{align}
\end{subequations}
The above equations characterize the meromorphic 
function $\psi(z,m)$ up to multiplication by a function of $q$, analytic
in the unit disk $|q|<1$.
\end{lemma}
These  follow from the definition of $\psi(z,w)$
and equations~\eqref{eq.xqinfty}, \eqref{eq.psi0} and 
\eqref{eq.psi01}-\eqref{eq.psi02}. Alternatively, they can be derived from 
equation~\eqref{eq.psiID} and the symmetries of the tetrahedron index 
$\ID$ given in \cite[Thm.3.2]{Ga:3D}.

For completeness, the next lemma summarizes the symmetries of $\IKD(m,e)$.

\begin{lemma}
\label{lem.IKD}
For all integers $m$ and $e$ we have the $\BZ/2$ and 
$\BZ/3$-invariance equations:
\begin{subequations}
\begin{align}
\label{eq.IKD2}
\IKD(m,e)(q) &= (-q)^{e+m} \IKD(-e,-m)(q) \\
\label{eq.IKD3}
\IKD(m,e)(q) &= (-q)^{e} \IKD(-e-m,m)(q) = (-q)^{e+m} \IKD(e,-e-m)(q) \,. 
\end{align}
\end{subequations}
As a consequence, we have another $\BZ/2$-invariance equation
\be
\label{eq.IKD22}
\IKD(m,e)=\IKD(-m,m+e) \,.
\ee
\end{lemma}

\noindent
These follow from equations~\eqref{eq.psi0me} 
and~\eqref{eq.psi022}-\eqref{eq.psi03}. Additionally, they follow from
equations~\eqref{eq.3d2} and the symmetries of the tetrahedron index
$\ID(m,e)$ given in \cite[Thm.3.2]{Ga:3D}.

\subsection{The pentagon identity and the Pachner 2--3 move}
\label{sub.pentagon}

Let us recall the pentagon identity for the tetrahedron index $\ID$
from~\cite[Thm.3.7]{Ga:3D}.
\be
\label{eq.pentagon.ID}
\ID(m_1-e_2,e_1) \ID(m_2-e_1,e_2) =\sum_{e_3 \in \BZ}
q^{e_3} \ID(m_1,e_1+e_3) \ID(m_2,e_2+e_3) \ID(m_1+m_2,e_3)
\ee
for all integers $m_1,m_2,e_1,e_2$. Replacing $q$ by $q^2$ 
in~\eqref{eq.pentagon.ID} and using equation~\eqref{eq.3d2}, it follows that 
the tetrahedron index $\IKD$ satisfies the equation
\begin{multline}
 \label{eq.pentagon1}
\IKD(m_1-e_2,e_1) \IKD(m_2-e_1,e_2) \\=\sum_{e_3 \in \BZ}
(-q)^{-e_3} \IKD(m_1,e_1+e_3) \IKD(m_2,e_2+e_3) \IKD(m_1+m_2,e_3) \,.
\end{multline}
To remove the factor $(-q)^{e_3}$ in the above equation, we apply
equation~\eqref{eq.IKD3} to the two terms of the left hand side and to the
last term in the right hand side, and obtain
\begin{multline*}
 \IKD(e_1,-e_1+e_2-m_1) \IKD(e_2,e_1-e_2-m_2) \\=\sum_{e_3 \in \BZ}
\IKD(m_1,e_1+e_3) \IKD(m_2,e_2+e_3) \IKD(e_3,-e_3-m_1-m_2) \,.
\end{multline*}
Setting 
$$
x=e_1, \quad y=-e_1+e_2-m_1, \quad u=e_2, \quad v=e_1-e_2-m_2,\quad e_3=-z,
$$
%
we rewrite the latter equality as
$$
\IKD(x,y) \IKD(u,v) =\sum_{z \in \BZ}
\IKD(u-x-y,x-z) \IKD(-z,z+v+y) \IKD(-u-v+x,u-z) \,.
$$
Applying equation~\eqref{eq.IKD22} to both functions on the left hand side 
of the above equation, and after a linear change of variables
we obtain
$$
\IKD(x,y) \IKD(u,v) =\sum_{z \in \BZ}
\IKD(-u-y,-x-z) \IKD(-z,z+x+y+u+v) \IKD(-x-v,-u-z) \,.
$$
Applying equation~\eqref{eq.IKD22} to the first and third terms of the
right hand side of the above equation, we obtain
$$
\IKD(x,y) \IKD(u,v) =\sum_{z \in \BZ}
\IKD(u+y,-x-y-u-z) \IKD(-z,z+x+y+u+v) \IKD(x+v,-x-u-v-z) \,.
$$
Finally, changing the summation variable $z\mapsto z-x-y-u-v$, we obtain
equation
\be
\label{eq.pentagon3}
\IKD(x,y) \IKD(u,v) =\sum_{z \in \BZ}
\IKD(u+y,v-z) \IKD(x+y+u+v-z,z) \IKD(x+v,y-z) \,.
\ee
which coincides with a special (constant) form of the beta pentagon 
equation~\cite[Eqn.(2)]{Kashaev:beta} for the LCA group $\BZ$.  Using the 
fact that the beta pentagon relation is stable under Fourier transformation, 
see~\cite[Sec.2]{Kashaev:beta}, we also conclude that the function
\be
\label{eq:phi-from-3D}
\phi(x,y):=\sum_{m,n\in\BZ}x^m y^{-n}\IKD(m,n)=\psi^0(x,1/y)
\ee
satisfies the constant beta pentagon equation for the circle LCA group 
$\mathbb{T}=\{z\in\BC\ \vert\ |z|=1\}$:
\be
\label{eq.pentagonK}
\phi(x,y) \phi(u,v) =\int_{\mathbb{T}}
\phi(uy,v/z) \phi(xyuv/z,z) \phi(xv,y/z)\frac{\operatorname{d}\! z}{2\pi i z}
\,.
\ee
Equivalently, the function $\psi^0(z,w)$, thought of as a distribution on 
$\BT^2$, satisfies the integral identity
\be
\label{eq.pentagonKpsi0}
\psi^0\left(x,y\right)
\psi^0\left(u,v\right) = \int_{\mathbb{T}}
\psi^0\left(\frac uy,\frac vz\right) 
\psi^0\left(\frac{xuz}{yv},z\right) 
\psi^0\left(\frac xv,\frac yz\right)
\frac{\operatorname{d}\! z}{2\pi i z}
\ee
where all variables belong to the unit circle $\mathbb{T}$. 
The distributional interpretation of   $\psi^0(z,w)$ means that its 
restriction to $\BT^2$ should be obtained by approaching $\BT^2$ from the 
domain~\eqref{eq.psi0R}, which is the domain of absolute convergence of 
the double series~\eqref{eq.psi0me}, i.e.,
\begin{equation}
\psi^0(z,w)=\psi^0_{+0,+0}(z,w):=\lim_{\substack{\alpha,\gamma\to 0\\
\alpha>0,\gamma>0,\alpha+\gamma<\pi}}\psi^0_{\alpha,\gamma}(z,w),\quad |z|=|w|=1,
\end{equation}
where
\begin{subequations}
\begin{align}
\label{eq.psi0ac1}
\psi^0_{\a,\g}(z,w) & := \psi^0\left(z (-q)^{-\frac{\a+\g}{\pi}},
w (-q)^{-\frac{\a}{\pi}}\right) \\ 
\label{eq.psi0ac2}
& 
= c(q) \,
G_q(z (-q)^{\frac{\b}{\pi}}) \, G_q(w^{-1} (-q)^{\frac{\a}{\pi}}) \, 
G_q(z^{-1}w (-q)^{\frac{\g}{\pi}})
\end{align}
\end{subequations}
where $\alpha+\beta+\gamma =\pi$.
The positive parameters $\alpha,\beta,\gamma$ satisfying the condition  
$\a+\b+\g=\pi$ (which we call pre-angle structure) can be identified with 
the dihedral angles of a positively oriented ideal hyperbolic tetrahedron. 
These angles are placed on the edges of a tetrahedron with ordered vertices 
according to Figure~\ref{f.angles}.

\begin{figure}[!htpb]
\begin{center}
\begin{tikzpicture}[scale=1.3,baseline=-3]
\draw (0,0) circle  (1);
\draw (90:1)--(0,0)--(-30:1) (0,0)--(-150:1);
\draw (90:.5) node[fill=white]{\small $\alpha$};
\draw (-90:1) node[fill=white]{\small $\alpha$};
\draw (-30:.5) node[fill=white]{\small $\beta$};
\draw (150:1) node[fill=white]{\small $\beta$};
\draw (-150:.5) node[fill=white]{\small $\gamma$};
\draw (30:1) node[fill=white]{\small $\gamma$};
\draw (90:0.1) node[right]{\small 0};
\draw (90:1.1) node[above]{\small 1};
\draw (-30:1.1) node[right]{\small 2};
\draw (-150:1.1) node[left]{\small 3};
\end{tikzpicture}
\caption{The angles of an ideal tetrahedron with ordered vertices.}
\label{f.angles}
\end{center}
\end{figure}
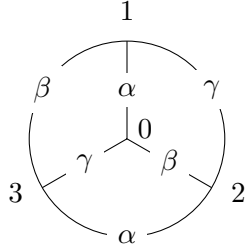


Moreover, the constant distributional beta pentagon 
identity~\eqref{eq.pentagonKpsi0} is a constant limit of the analytically 
continued non-constant identity

\be
\label{eq.pentagonKpsi0nc}
\psi^0_{\alpha_3,\gamma_3}\left(x,y\right) 
\psi^0_{\alpha_1,\gamma_1}\left(u,v\right) =\int_{\mathbb{T}}
\psi^0_{\alpha_0,\gamma_0}\left(\frac uy,\frac vz\right) 
\psi^0_{\alpha_2,\gamma_2}\left(\frac{xuz}{yv},z\right) 
\psi^0_{\alpha_4,\gamma_4}\left(\frac xv,\frac yz\right)
\frac{\operatorname{d}\! z}{2\pi i z}
\ee 
where $(\a_i,\b_i,\g_i)$ are pre-angle structures on five tetrahedra $T_i$ 
for $i=0,\dots,4$ which are compatible, i.e., satisfy the linear relations
\be
\label{eq.5angles}
\a_1=\a_0+\a_2, \quad \a_3=\a_2+\a_4, \quad \g_1=\g_0+\a_4, \quad 
\g_3=\a_0+\g_4, \quad \g_2=\g_1+\g_3 \,.
\ee
Notice that a compatible angle structure satisfies the balancing condition for
the interior edge $13$: $\b_0+\g_2+\b_4=2\pi$.
 Such an identity for Faddeev's quantum dilogarithm
appeared in~\cite[Prop.1]{AK:new}.


\begin{figure}[!hptb]
\begin{center}
\begin{tikzpicture}[scale=.7,baseline]
\draw[fill=gray!9] (2,0)--(0,2)--(-2,0)--(0,-2)--cycle;
\draw (0,2)--(.7,-.7)--(0,-2);
\draw (-2,0)--(.7,-.7)--(2,0);
\draw[dashed] (-2,0)--(2,0);
\draw (2,0) node[right]{$2$};
\draw (-2,0) node[left]{$0$};
\draw (0,-2) node[below]{$1$};
\draw (0,2) node[above]{$3$};
\draw (.7,-.7) node[above left]{$4$};
\end{tikzpicture}\quad
=\quad
\begin{tikzpicture}[scale=.7,baseline]
\draw[fill=gray!9] (2,0)--(0,2)--(-2,0)--(0,-2)--cycle;
\draw (0,2)--(.7,-.7)--(0,-2);
\draw (-2,0)--(.7,-.7)--(2,0);
\draw[dashed] (-2,0)--(2,0);
\draw[dotted] (0,2)--(0,-2);
\draw (2,0) node[right]{$2$};
\draw (-2,0) node[left]{$0$};
\draw (0,-2) node[below]{$1$};
\draw (0,2) node[above]{$3$};
\draw (.7,-.7) node[above left]{$4$};
\end{tikzpicture}
\end{center}
\caption{The ordered 2--3 Pachner move.}
\label{f.23move}
\end{figure}
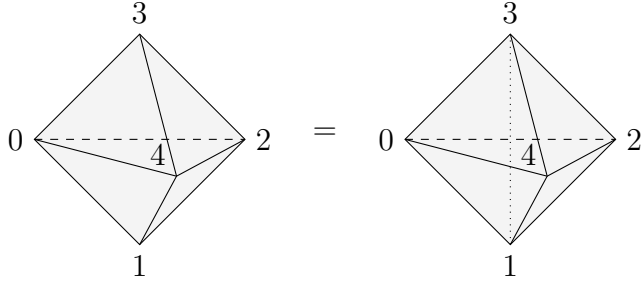


\section{The state-integral}
\label{sub:state-integral}

\subsection{Definition of the state-integral}
\label{sub.def}

Fix an ideal triangulation $\calT$ of an oriented 3-manifold with $N$ 
tetrahedra $T_i$ for $i=1,\dots, N$. The invariant is defined as follows:

\begin{itemize}
\item[(a)] 
Assign variables $x_i$ for $i=1,\dots,N$ to $N$ edges of $\calT$. 
\item[(b)] 
Choose a strictly positive pre-angle structure $\th=(\alpha,\b,\g)$ 
at each tetrahedron. Here, $\alpha$ is the angle of the $01$ and $23$ 
edges, $\b$ is the angle of the $02$ and $13$ edges, and $\g$ is 
the angle of the $03$ and $12$ edges. The angles are normalized so that at each
tetrahedron, their sum is $\pi$.
\item[(c)]
The weight of a tetrahedron $T$ is given by
\begin{subequations}
\begin{align}
\label{eq.B1}
B(T,x,\th) &= \psi^0\left(
(-q)^{-\frac{\a+\g}{\pi}}\frac{X_\a}{X_\g}, 
(-q)^{-\frac{\a}{\pi}} \frac{X_\b}{X_\g}  \right) \\
\label{eq.B2}
&= c(q) \, 
G_q\left((-q)^{\frac{\a}{\pi}} \frac{X_\g}{X_\b} \right)
G_q\left((-q)^{\frac{\b}{\pi}} \frac{X_\a}{X_\g} \right)
G_q\left((-q)^{\frac{\g}{\pi}} \frac{X_\b}{X_\a} \right)
\end{align}
\end{subequations}
where 
$$
X_\a := x_{01} x_{23}, \quad X_\b:=x_{02} x_{13}, \quad X_\g:=x_{03}x_{12},
$$
and $x_{ij}$ is the variable at edge $ij$ of the tetrahedron.
\item[(d)]
Define
\be
\label{eq.Ipre}
\Ipre_{\calT,\th}(q) = \int_{\BT^N} \prod_{i=1}^N B(T_i,x,\th) d\mu(x)
\ee
where $d\mu(x)=(2 \pi \im)^{-N}\prod_{i=1}^N dx_i/x_i$ is the normalized
Haar measure on $\BT^N$.
\end{itemize}

Recall the exponent matrices $(\Abar|\Bbar|\Cbar)$ of the edge gluing equations
of $\calT$~\cite{Th,NZ,snappy}. These are $N \times N$ matrices with integer 
entries which determine the gluing equations
\be
\label{eq.NZ}
\prod_{j=1}^N z_j^{\Abar_{i,j}} (z_j')^{\Bbar_{i,j}} 
(z_j'')^{\Cbar_{i,j}} =1, \qquad i=1,\dots,N
\ee
where $z'=1/(1-z)$ and $z''=1-1/z$, and as usual we have $z z' z'' =-1$.

For a vector $x=(x_1,\dots,x_N)$ of nonzero complex numbers, and a vector
$v=(v_1,\dots,v_N)$ of integers, define $x^v = \prod_{i=1}^N x_i^{v_i}$.
Also, for a matrix $A$, let $A_i$ denote its $i$th column. 

The next proposition implies that the integral~\eqref{eq.Ipre} 
(and even the integrand) depends on only
the Neumann--Zagier matrices of the gluing equations of the triangulation
$\calT$. 

\begin{proposition}
\label{prop.ABC}
With the above notation and for $i=1,\dots,N$, we have:
\be
\label{eq.BTABC}
B(T_i,x,\th) = c(q) \, 
G_q\left((-q)^{\frac{\a_i}{\pi}} x^{(\Cbar-\Bbar)_i} \right)
G_q\left((-q)^{\frac{\b_i}{\pi}} x^{(\Abar-\Cbar)_i} \right)
G_q\left((-q)^{\frac{\g_i}{\pi}} x^{(\Bbar-\Abar)_i} \right)
\ee
It follows that $I_{\calT,\th}(q)$ depends on only the matrices $\Abar$, $\Bbar$,
$\Cbar$ and $\th$. 
\end{proposition}

\begin{proof}
Let $e_i \in \BZ^N$ denote the $i$th coordinate vector (1 in position $i$
and 0 otherwise). It suffices to show that
\be
\label{eq.Xgb}
\left(\frac{X_\g}{X_\b}\right)_i = x^{(\Cbar-\Bbar)e_i},
\qquad
\left(\frac{X_\a}{X_\g}\right)_i = x^{(\Abar-\Cbar)e_i},
\qquad
\left(\frac{X_\b}{X_\a}\right)_i = x^{(\Bbar-\Abar)e_i} \,.
\ee
This follows from the fact that the matrices $\Abar, \Bbar$ and $\Cbar$
indexed by $\text{edges} \times \text{tetrahedra}$ record the number of 
times a shape $z_j$ (resp., $z_j'$, $z_j''$) of the $j$th tetrahedron 
appears around an edge $e_i$.
\end{proof}

Our choice of $\th$ and equation~\eqref{eq.psi0R} imply that the 
integrand in \eqref{eq.Ipre} is an analytic function of $x \in \BT^N$, 
therefore the integral 
converges. We are interested in two affine vector subspaces of $\BC^{3N}$:
\begin{itemize}
\item
$\calA_{\calT}$, the space of complexified pre-angle structures, i.e., 
$\th \in \BC^{3N}$ such that $\a_i+\b_i+\g_i=\pi$ for $i=1,\dots,N$.
\item 
$\calB_{\calT} \subset \calA_{\calT}$, the affine subspace of $\calA_{\calT}$ that
consists of balanced complexified pre-angle structures, that is
the sum of the angles around each edge of $\calT$ is $2\pi$. The points
of $\calB$ are also known as complex-valued angle structures on $\calT$.
\end{itemize}
$\calA_{\calT}$ and $\calB_{\calT}$ are complex affine subspaces of
$\BC^{3N}$ of dimension $2N$ and $N+1$ respectively~\cite{LT}.

The integral in \eqref{eq.Ipre} extends to a meromorphic function of
$\th \in \calA_{\calT}$, regular when $\Re(\th) >0$. Our task is to show that 
this extension restricts to a meromorphic function on $\calB_{\calT}$. 
To do so, it will be convenient to parametrize $\calA_{\calT}$. This breaks the
symmetry in the definition of the integral, however it is a useful gauge
to draw conclusions. 

Consider a vector 
$\ve=(\ve_1,\dots,\ve_N) \in \BC^N$ and complex numbers $\mu,\lambda$ 
defined by
\be
\label{eq.abc}
\Abar \a + \Bbar \b + \Cbar \g 
=
\pi \mathbf{2} + \ve, \qquad \begin{matrix} 
\nu_\mu \cdot \a + \nu'_\mu \cdot \b + \nu''_\mu \cdot \g =\mu \\
\nu_\lambda \cdot \a + \nu'_\lambda \cdot \b + \nu''_\lambda \cdot \g =\lambda 
\end{matrix} 
\ee
where $\mu$ and $\lambda$ are the sums of the angles along the meridian and
the longitude curves, 
$\mathbf{2} \in \BC^N$  is the vector with all coordinates $2$ and
$\nu_\mu$, $\nu_\lambda$ and their primed versions are the vectors
of the meridian and longitude cusp equations. Note that $\ve_1+\dots+\ve_N=0$.
A \emph{quad} is a choice of a pair of opposite edges in a tetrahedron. 
Choosing a quad, allows one to eliminate the angle variable of those edges 
using the equation $\a_i+\b_i+\g_i=\pi$. Note that each tetrahedron has 3 quads,
hence $\calT$ has $3^N$ quads. If $Q$ is a system of $N$ quads obtained by 
choosing one quad from each tetrahedron of $\calT$, let $Q(\th) \in \BC^N$ 
be defined so that its $i$th component is the angle associated to the 
corresponding quad in $T_i$, i.e. $Q(\th)_i=\a_i$ if, for example, the quad 
of the $i$th tetrahedron is $(01),(23)$. In effect, $Q(\th)$ chooses one of
the 3 angles for every tetrahedron of $\calT$. Together with 
equations~\eqref{eq.abc}, we get a linear map
$$
\calA_{\calT} \to \BC^N \times \BC^N \times \BC^2, \qquad \th \mapsto 
(Q(\th),\ve,\mu,\lambda) \,.
$$
For $i,j \in \{1, \dots, N\}$, let $\pi_{i,j} : \BC^N \times \BC^N \times \BC^2
\to \BC^{N-1} \times \BC^{N-1} \times \BC^2$ denote the projection that
removes the $i$th entry of the first copy of $\BC^N$ and the $j$th entry
of the second copy of $\BC^N$. The next proposition describes a 
parametrization of $\calA_{\calT}$.

\begin{proposition}
\label{prop.quad}
For every $j$, there exists a system of quads $Q$ of $\calT$ and an $i$ 
such that the composition $\calA_{\calT} \to \BC^N \times \BC^N \times \BC^2 
\stackrel{\pi_{i,j}}\to \BC^{N-1} \times \BC^{N-1} \times \BC^2$ 
is an affine linear isomorphism.
\end{proposition}

\begin{proof}
Consider the standard system of quads $Q$ of $\calT$, i.e., the one that 
chooses the $(01),(23)$ edges of each tetrahedron of $\calT$.
It follows that $Q(\th)=\a$. Eliminating $\b$ from Equation~\eqref{eq.abc} 
we obtain that
\be
\label{eq.AB}
A \a + B \g = \pi\bf{2} + \ve + \nu
\ee
where
$$ 
A=\Abar-\Bbar, \qquad B=\Cbar-\Bbar, \qquad \nu = - \Bbar \bf{1}\pi \,.
$$
Consider the matrix $(A'|B')$ obtained from $(A|B)$ by replacing 
the $j$th row of $(A|B)$ with the peripheral cusp equation corresponding
to the meridian. Neumann--Zagier prove that $(A'|B')$ is the upper half of a 
symplectic $2N\times 2N$ matrix~\cite{NZ}. In fact we can take the first
row of the bottom half of the symplectic matrix to be the peripheral
cusp equation corresponding to the longitude. If $B'$ is invertible, then
we can solve for $\g$ from Equation~\eqref{eq.AB} and deduce that
$\th$ is determined by $(\a,\pi_j(\ve),\mu)$. Using the longitude cusp
equation, it follows that $\th$ is determined by 
$(\pi_i(\a),\pi_j(\ve),\mu,\lambda)$. 

When $B'$ is not invertible, using the fact that $(A'|B')$ is the
upper half of a symplectic matrix and \cite[Lem.A.3]{DG}, it follows that
we can always find a system of quads $Q$ for which the corresponding matrix $B'$
is invertible. The result follows.
\end{proof}

Without loss of generality, we can assume that Proposition~\ref{prop.quad} 
holds for
the standard system of quads, and that $i=j=N$. After a change of variables 
$x_i \to x_i/x_N$ for $i=1,\dots,N-1$ the integral $I_{\calT,\th}(q)$ reduces 
to an $N-1$ dimensional integral since the integrand is independent of $x_N$.

The next lemma shows that after a change of variables, $I_{\calT,\th}(q)$
is expressed as an integral whose contour (a product of tori, with radi
$|q|$ raised to linear forms of $\a$) depends only on $\a$ and whose 
integrand depends on only  $(\ve,\mu,\lambda)$. 

\begin{proposition}
\label{prop.ptheta}
With the above assumptions, there exists an edge column vector $\eta$
whose coordinates are affine linear forms in $\a$,
a contour $C_\a$, and affine linear forms $\nu_i, \nu_i', \nu_i''$ in the 
variables $\ve=(\ve_1,\dots,\ve_{N-1})$ and $\mu,\lambda$ such that after 
the change of variables $x_i=(-q)^{\eta_i} y_i$, we have
\small{
\begin{align}
\label{eq.Iptheta}
\Ipre_{\calT,\th}(q) &= c(q)^N 
\int_{C_\a} d\mu(y)
\\ \notag & \qquad 
\prod_{i=1}^N 
G_q\left((-q)^{\nu_i(\ve,\mu,\lambda)} y^{(\Cbar-\Bbar)_i} \right)
G_q\left((-q)^{\nu_i'(\ve,\mu,\lambda)} y^{(\Abar-\Cbar)_i} \right)
G_q\left((-q)^{\nu_i''(\ve,\mu,\lambda)} y^{(\Bbar-\Abar)_i} \right)
\end{align}
}
\end{proposition}

\begin{proof}
Recall equations~\eqref{eq.Xgb} and \eqref{eq.AB}. If $x=(-q)^\eta y$ then
$$
x^{(\Cbar-\Bbar)e_i} 
= (-q)^{\a_i + e_i^T (\Cbar-\Bbar)^T \eta} y^{(\Cbar-\Bbar)e_i}=
(-q)^{\a_i + e_i^T B^T \eta} y^{B e_i}   
$$
and likewise for the  cyclic permutations. Using equation~\eqref{eq.BTABC}, 
and the above equalities, we are looking for an edge vector $\eta$
such that $\a+B^T \eta$, $\b+(A-B)^T\eta$ and $\g-A^T\eta$ depend on only
$\ve,\mu,\lambda$. Since the sum of these three vectors is constant, it
suffices to find $\eta$ such that $\a-B^T \eta$ and $\g-A^T\eta$ 
depend on only $\ve,\mu,\lambda$. Using the fact that $B'$ is invertible, 
we can take $\eta=(B'^T)^{-1}\a$. Solving for $\g$ from equation~\eqref{eq.AB},
we obtain that
$$
\g-A'^T\eta=B'^{-1}(\pi\mathbf{2}+\ve+\nu)-(B')^{-1}A'\a-(A')^T((B')^T)^{-1}\a =
B'^{-1}(\pi\mathbf{2}+\ve+\nu)\,.
$$
The last equality follows from the fact that $(A'|B')$ is the upper half 
of a symplectic matrix, hence $(B')^{-1}A'$ is symmetric. Finally observe
that $|x_i|=1$, hence $|y_i|=|q|^{\eta_i}$ where $\eta_i$ are linear forms in
$\a$. Hence, the $y$-contour $C_\a$ is a product of tori whose radii depend
linearly on $\a$. This completes the proof. 
\end{proof}

\begin{remark}
\label{rem.ptheta2}
With the assumptions of Proposition~\ref{prop.ptheta}, there exists an edge 
column vector $\eta'$ whose coordinates are affine linear forms in 
$\alpha, \mu,\lambda$, a contour $C$, and affine linear forms 
$\xi_i, \xi_i', \xi_i''$ 
in the variables $\ve=(\ve_1,\dots,\ve_{N-1})$, $\mu$ and $\lambda$ such that 
after the change of variables $y_i=(-q)^{\eta'_i} z_i$, we have
\small{
\begin{align}
\label{eq.Iptheta2}
\Ipre_{\calT,\th}(q) &= c(q)^N 
\int_{C} d\mu(z)
\\ \notag & \qquad 
\prod_{i=1}^N 
G_q\left((-q)^{\xi_i(\ve,\mu,\lambda)} z^{(\Cbar-\Bbar)_i} \right)
G_q\left((-q)^{\xi_i'(\ve,\mu,\lambda)} z^{(\Abar-\Cbar)_i} \right)
G_q\left((-q)^{\xi_i''(\ve,\mu,\lambda)} z^{(\Bbar-\Abar)_i} \right)
\end{align}
}
where $\xi_i(0,\mu,\lambda)$, $\xi'_i(0,\mu,\lambda)$ and 
$\xi''_i(0,\mu,\lambda)$ are $\BZ$-linear combinations of 
$1,\mu/pi,\lambda/(2 \pi)$.
This follows from the symplectic properties of the Neumann--Zagier 
matrices~\cite{NZ} (compare also with the matrices $(A'|B')$ 
in equation~\eqref{eq.A'B'} below).
\end{remark}


The contours in equations~\eqref{eq.Iptheta} and 
\eqref{eq.Iptheta2} depend on a positive pre-angle structure, but the 
integrals are independent of the choice of the pre-angle structure. When 
we balance, i.e., set $\ve=0$, there are two posibilities: either we 
can move the contour by a small isotopy in order to avoid the singularities 
of the integrand, or we cannot do so. In the former case, the new contour 
is canonically defined from the old contour. In the latter case, we apply 
the residue theorem to change the integration contour, and the 
residue contribution is either finite or infinite. In the latter case, by 
definition, our meromorphic function is infinity.

\begin{remark}
\label{rem.bad}
An example of an integral where the
latter case occurs is the following integral
\be
\label{eq.bad}
I_\ve(C):=\int_C \frac{dz}{z} G_q(z) G_q((-q)^\ve z^{-1})
\ee
where $C$ is the contour $|z|=1-\delta$ for $\delta>0$ small. 
The singularities of the integrand
are $z \in q^{-\BN}$ and $z \in (-q)^\ve q^\BN$. When $\ve$ approaches zero,
the contour is pinched from two sides by the singularities at $z=(-q)^\ve$ 
and $z=1$. To avoid this pinching, by applying the residue theorem, we move 
the contour of integration $C$ to the other side of $1$, i.e. the contour 
$C'$  given by  $|z|=1+\delta$:  
\be
I_\ve(C)=I_\ve(C')-2\pi i \operatorname{Res}_{z=1} 
\left(\frac{1}{z} G_q(z) G_q((-q)^\ve z^{-1})\right)=
I_\ve(C')+2\pi i \frac{\poc{-q;q}{\infty}}{\poc{q;q}{\infty}}G_q((-q)^\ve)
\,.
\ee 
Even though $I_\ve(C')$ is regular at $\ve=0$, the contribution 
from the residue is singular because of the simple pole of $G_q((-q)^\ve)$
at $\ve=0$. Thus, we conclude that $I_0(C)=\infty$.
\end{remark}

The next proposition defines the meromorphic function that appears in
Theorem~\ref{thm.1}.

\begin{proposition}
\label{prop.balance}
Setting $\ve=0$ (i.e., balancing the edges), and assuming that we find
a contour $C$, we obtain a meromorphic function 
of $(\emu,\elambda):=((-q)^{\mu/\pi},(-q)^{\lambda/(2\pi)}) \in (\BC^*)^2$
\small{
\begin{align*}
I_{\calT,\emu,\elambda}(q) &= c(q)^N 
\int_C d\mu(z)
\\ & \qquad 
\prod_{i=1}^N 
G_q\left((-q)^{\xi_i(0,\mu,\lambda)} z^{(\Cbar-\Bbar)_i} \right)
G_q\left((-q)^{\xi_i'(0,\mu,\lambda)} z^{(\Abar-\Cbar)_i} \right)
G_q\left((-q)^{\xi_i''(0,\mu,\lambda)} z^{(\Bbar-\Abar)_i} \right) \,.
\end{align*}
}
The above integral is absolutely convergent.
\end{proposition}

Let us phrase our previous proposition in more invariant language.
Recall the complex torus $\BT_M=H^1(\pt M,\BC^*)$. Define the map
\be
\label{eq.fiberpi}
\varpi_{\calT}: \calB_{\calT} \to \BT_M, 
\qquad \delta \mapsto (-q)^{\text{hol}(\delta)/\pi}
\ee
for a simple closed curve $\delta$ of $\partial M$, where 
$\text{hol}(\delta)$ denotes the angle holonomy along $\delta$. 
Proposition~\ref{prop.balance} states that the restriction $\Ibal_{\calT}$ of 
the meromorphic function $\Ipre_{\calT}$ on $\calB$ is the pullback of 
a meromorphic function $I_\calT$ on $\BT_M$. In other words, the following
diagram commutes:

\be
\label{eq.CD}
\xymatrix{ \calA_{\calT} \ar[rd]_{\Ipre_{\calT}} & 
\calB_{\calT} 
\ar@{_{(}->}[l] \ar@{->>}[r]^{\varpi_{\calT}} \ar[d]_{\Ibal_{\calT}} 
& \BT_M \ar[ld]_{I_{\calT}} \\
& \BC
}
\ee

\subsection{Identification with the 3D-Index of Dimofte--Gaiotto--Gukov}
\label{sub.identify}

In this sub-section we discuss the Laurent expansion of the meromorphic 
function $I_{\calT,\emu,\elambda}(q)$ on the real torus $|\emu|=|\elambda|=1$,
under the assumption that $\calT$ supports a strict  angle structure.
As we will show, the coefficients of the Laurent series are the 
3D-index of Dimofte--Gaiotto--Gukov. This will conclude the 
proof of Theorem~\ref{thm.2}.

The next lemma (for $h=0$) is based on the idea that the upper half part of 
a symplectic matrix with integer entries is a pair of coprime matrices.

\begin{lemma}
\label{lem.symplectic}
Suppose $M=\begin{pmatrix} A & B \\ C & D \end{pmatrix}$ 
is a symplectic matrix with integer entries where $A,B,C,D$ are $N \times N$ 
matrices.
Let $(A'|B')$ denote the upper $(N+h)\times 2N$
part of $M$ and consider $r,s \in \BZ^{N}$ that satisfy the equation
$$
A' r + B' s = \begin{pmatrix} 0 \\ \nu \end{pmatrix}
$$ for a vector $\nu \in \BZ^h$. 
Then, there exists $k \in \BZ^{N-h}$ such that
$$
r=-B'^T \begin{pmatrix} \nu \\ k \end{pmatrix}, \qquad 
s=A'^T  \begin{pmatrix} \nu \\ k \end{pmatrix}\,.
$$
\end{lemma}

\begin{proof}
We apply the symplectic matrix 
$M=\begin{pmatrix} A & B \\ C & D \end{pmatrix}$ to the vector 
$\begin{pmatrix} r \\ s \end{pmatrix}$ and define $k \in \BZ^{N-h}$ so that
$k':=\begin{pmatrix} \nu \\ k \end{pmatrix} \in \BZ^N$ satisfies
$ M \begin{pmatrix} r \\ s \end{pmatrix} = 
\begin{pmatrix} 0 \\ k' \end{pmatrix}$.
Then, 
$$
\begin{pmatrix} r \\ s \end{pmatrix} = M^{-1}
\begin{pmatrix} 0 \\ k' \end{pmatrix} =
\begin{pmatrix} D^T & -B^T \\ -C^T & A^T \end{pmatrix}
\begin{pmatrix} 0 \\ k' \end{pmatrix} =
\begin{pmatrix} -B^T k' \\ A^T k' \end{pmatrix} \,.
$$
The result follows.
\end{proof}

\begin{proof}(of Theorem~\ref{thm.2})
Fix an ideal triangulation $\calT$ of $M$ with $N$ tetrahedra and let
$\Abar, \Bbar$ and $\Cbar$ denote the Neumann--Zagier matrices describing
gluing equations~\eqref{eq.NZ}. If we eliminate the shape $z_i'=-1/(z_i z_i'')$,
we obtain the gluing equations in the form:
\be
\label{eq.NZ2}
\prod_{j=1}^N z_j^{A_{i,j}} 
(z_j'')^{B_{i,j}} =(-1)^{\nu_i}, \qquad i=1,\dots,N
\ee
where 
$$
A=\Abar-\Bbar, \qquad B=\Cbar-\Bbar \,.
$$ 
Notce that $\Abar-\Cbar=A-B$, $\Bbar-\Cbar=-B$. Fix a strict angle structure
$\th$ and use Proposition~\ref{prop.ABC} to write the 
integrand of $\Ipre_{\calT,\th}$ as follows:
\begin{align*}
\prod_{i=1}^N B(T_i,x,\th) &= \prod_{i=1}^N \psi^0\left(
(-q)^{-\frac{\a_i+\g_i}{\pi}}\frac{X_{\a_i}}{X_{\g_i}}, 
(-q)^{-\frac{\a_i}{\pi}} \frac{X_{\b_i}}{X_{\g_i}}  \right) \\
&= \prod_{i=1}^N \psi^0\left(
(-q)^{-\frac{\a_i+\g_i}{\pi}} x^{(\Abar-\Cbar)_i}, 
(-q)^{-\frac{\a_i}{\pi}} x^{(\Bbar-\Cbar)_i}  \right) \\
&=\prod_{i=1}^N \psi^0\left(
(-q)^{-\frac{\a_i+\g_i}{\pi}} x^{(A-B)_i}, 
(-q)^{-\frac{\a_i}{\pi}} x^{-(B)_i}  \right) \,.
\end{align*}
Use equation~\eqref{eq.psi0me} to expand the integrand of 
$\Ipre_{\calT,\th}$ as a convergent series on the torus $|x_i|=1$ 
(for $i=1,\dots,N$): 
$$
\prod_{i=1}^N B(T_i,x,\th) = \sum_{r, s \in \BZ^N}
\IKD(r_1,s_1) \dots \IKD(r_N,s_N) x^{\sum_i (A-B)_i s_i }
x^{-\sum_i (B)_i r_i } (-q)^{-\frac{1}{\pi}\sum_i (\a_i+\g_i)s_i+\a_i r_i}
$$
where $r=(r_1,\dots,r_N) \in \BZ^N$ and $s=(s_1,\dots,s_N) \in \BZ^N$.
Now, we can compute the absolutely convergent integral $\Ipre_{\calT,\th}$
by integrating over the torus. After interchanging summation and integration
(justified by uniform convergence) and applying the residue theorem we
obtain a sum over $r,s \in \BZ^N$ such that $(A-B)s-Br=0$:
\begin{align*}
\Ipre_{\calT,\th}&=\int_{\BT^N} d\mu(x) \prod_{i=1}^N B(T_i,x,\th)
\\
&=\sum_{r, s}
\IKD(r_1,s_1) \dots \IKD(r_N,s_N) 
(-q)^{-\frac{1}{\pi}\sum_i (\a_i+\g_i)s_i+\a_i r_i} \int_{\BT^N} d\mu(x)
x^{\sum_i (A-B)_i s_i-\sum_i (B)_i r_i } \\
&= \sum_{r, s: (A-B)s-Br=0}
\IKD(r_1,s_1) \dots \IKD(r_N,s_N) 
(-q)^{-\frac{1}{\pi}\sum_i (\a_i+\g_i)s_i+\a_i r_i} \,.
\end{align*}
Collecting further terms whose meridian and (half) longitude holonomy is a 
fixed integer, we obtain the uniformly convergent sum:
\be
\label{eq.Ithm2}
\Ipre_{\calT,\th}(q) = 
\sum_{m, e \in \BZ} (-q)^{\frac{m\mu}{\pi}}(-q)^{\frac{e\lambda}{2\pi}}
\sum_{r, s} \IKD(r_1,s_1) \dots \IKD(r_N,s_N)
\ee 
where $r,s \in \BZ^N$ satisfy the equation 
\be
\label{eq.A'B'}
A's + B'(-r-s) = \begin{pmatrix} 0 \\ \nu \end{pmatrix}, \qquad
\nu = \begin{pmatrix} m \\ e \end{pmatrix} 
\ee
for the matrix $(A'|B')$ (where $A',B'$ are $(N+1)\times N$ matrices)
obtained from $(A|B)$ after we remove any one row of it and replace it with 
two rows of the meridian and half-longitude monodromy. 
Neumann--Zagier~\cite{NZ} prove that $(A'|B')$ can be completed to a 
symplectic matrix. Using this, and Lemma~\ref{lem.symplectic}, it follows
that there exists $k \in \BZ^{N-1}$ such that 
$$
s=-B'^T \begin{pmatrix} \nu \\ k \end{pmatrix},
\qquad -r-s=A'^T \begin{pmatrix} \nu \\ k \end{pmatrix} \,.
$$ 
Let $a_i$ and $b_i$ for $i=1,\dots,N$ denote the $i$th column of $A'$ and 
$B'$ respectively and let $k'=\begin{pmatrix} \nu \\ k \end{pmatrix}$.
Using the above and 
equations~\eqref{eq.IKD2}, \eqref{eq.IKD22}, we obtain that
\begin{align*}
\IKD(r_i,s_i) &=\IKD(-a_i \cdot k' + b_i \cdot k' , -b_i \cdot k')
=\IKD(-b_i \cdot k', a_i \cdot k') 
\end{align*}
for $i=1,\dots,N$. Combined with equation~\eqref{eq.Ithm2}, this gives
\begin{align*}
\Ipre_{\calT,\th}(q) &= 
\sum_{m, e \in \BZ} (-q)^{\frac{m\mu}{\pi}}(-q)^{\frac{e\lambda}{2\pi}}
\sum_{k \in \BZ^{N-1}} \prod_{i=1}^N \IKD(r_i,s_i)(q) \\
&= 
\sum_{m, e \in \BZ} (-q)^{\frac{m\mu}{\pi}}(-q)^{\frac{e\lambda}{2\pi}}
\sum_{k \in \BZ^{N-1}} \prod_{i=1}^N \IKD(-b_i \cdot k', a_i \cdot k')(q) \\
&= 
\sum_{m, e \in \BZ} (-q)^{\frac{m\mu}{\pi}}(-q)^{\frac{e\lambda}{2\pi}}
\sum_{k \in \BZ^{N-1}} \prod_{i=1}^N (-q)^{v \cdot k'} 
\ID(-b_i \cdot k', a_i \cdot k')(q^2) \,.
\end{align*}
where the last equality follows from equation~\eqref{eq.3d2}
and $v=(1,\dots,1,m,e) \in \BZ^{N+1}$. 
The latter sum coincides with the 3Dindex of~\cite{DGG1}; see 
also~\cite[Sec.4.5]{GHRS}. This completes the proof of Theorem~\ref{thm.2}.
\end{proof}

\subsection{Invariance under 2--3 Pachner moves}
\label{sub.23}

Next we prove the invariance of the meromorphic function $I_{\calT}$
under 2--3 Pachner moves. Consider two ideal triangulations $\calT$ and 
$\widetilde\calT$ with $N$ and $N+1$ tetrahedra
that are related by a 2--3 Pachner move as in Figure~\ref{f.23move}
and choose $\th$ and $\widetilde\th$ compatible positive angle structures
on $\calT$ and $\widetilde\calT$ that satisfy equations~\eqref{eq.5angles}. 
In particular, this means that the sum of angles around the interior edge of 
the 3 tetrahedra is $2 \pi$. 


Note that $\widetilde\th$ determines $\th$, but not
vice-versa. Let us denote the linear map $\widetilde\th \mapsto \th$ by
$\th=m^2_3(\widetilde\th)$. The commutative diagram~\eqref{eq.CD} gives a 
commutative diagram

\be
\xymatrix{ \calA_{\widetilde\calT} \ar@{->>}[d]_{m^2_3} & 
\calB_{\widetilde\calT} 
\ar@{_{(}->}[l] 
\ar@{->>}[r]^{\varpi_{\widetilde\calT}} 
\ar@{->>}[d]_{m^2_3}
& \BT_M \ar@{=}[d] \\
\calA_{\calT} & 
\calB_{\calT} 
\ar@{_{(}->}[l] 
\ar@{->>}[r]^{\varpi_{\calT}} 
& \BT_M 
}
\ee
Let $\widetilde\th$ and $\th$ be positive pre-angle structures.
We claim that 
\be
\label{eq.5term}
\Ipre_{\calT,\th}(q)=\Ipre_{\widetilde\calT,\widetilde\th}(q) \,.
\ee
This follows by separating the integration variable of the inner edge of the 
2--3 move in the integral $\Ipre_{\widetilde\calT,\widetilde\th}(q)$ and
applying the pentagon identity~\eqref{eq.pentagonKpsi0nc}
to that variable. Since a meromorphic function is uniquely determined
by its values on positive pre-angle structures and the map $m^2_3:
\calA_{\widetilde\calT} \to \calA_{\calT}$ is onto, it follows that
$\Ipre_{\widetilde\calT} = \Ipre_{\calT} \circ m^2_3$. Now restrict to
$\calB_{\widetilde\calT}$, use the above commutative diagram and the fact that
$m^2_3$ is onto. It follows that 
$\Ibal_{\widetilde\calT} = \Ibal_{\calT} \circ m^2_3$. Using once again the
commutative diagram and the fact that $\varpi_{\calT}$ and 
$\varpi_{\widetilde\calT}$ are onto, it follows that
$I_{\widetilde\calT} = I_{\calT}$.

\subsection{The singularities of $I_{\calT}(q)$}
\label{sub.singularities}

The singularities of the integrand of $I_{\calT,\emu,\elambda}(q)$ are
given by Lemma~\ref{lem.Gq}. To determine the singularities of the
meromorphic function $I_\calT$, perform one integral at a time and use
the next lemma.

\begin{lemma}
\label{lem.mero}
Suppose $f(z)$ is a meromorphic function of $z$ with singularities
on $q^{-\BN}$. Fix positive integers $a_1,\dots,a_p >0$ and 
$b_1,\dots,b_n >0$, let
$s=(s_1,\dots,s_p)$, $t=(t_1,\dots,t_n)$ and consider the integral:
$$ 
F(s,t)=\int_\calC \prod_{i=1}^p f(s_i z^{a_i}) \prod_{j=1}^n f(t_j z^{-b_j})
\frac{dz}{z}  
$$
where $\calC$ is a contour that separates $q^{-1-\BN}$ from $q^{\BN}$.
Then, $F(s,t)$ is a meromorphic function of $(s,t)$ with singularities at 
a subset of
\be
\label{eq.sitj}
\{ (s,t) \,\,| \,\, s_i^{b_j} t_j^{a_i} \in q^{-a_i \BN -b_j \BN}
\,\, \text{for some} \,\, 1 \leq i \leq p, 1 \leq j \leq n\}
\ee
\end{lemma}

\begin{proof}
The singularities of the integrand is the set $\Sigma^-(s) \cup \Sigma^+(t)$
where
$$
\Sigma^-(s) = \cup_{i=1}^p \{ z^{a_i} \in s_i^{-1} q^{-\BN} \} \,,
\qquad 
\Sigma^+(t)= \cup_{j=1}^n
\{ z^{b_j} \in t_jq^{\BN} \} \,. 
$$
As long as $\Sigma^-(s)\cup\Sigma^+(t)$ does not touch the contour $\calC$, 
$F(s,t)$ is regular. It follows that if $F(s,t)$ is singular when pinching 
occurs, in other words we must have $z^{a_i} = s_i^{-1} q^{-k}$ and 
$z^{b_j} = t_j q^{l}$ for some $i,j$ and some $k,l \in \BN$. Thus
$z^{a_i b_j} =  (s_i^{-1} q^{-k})^{b_j} = (t_j q^{l})^{a_i}$. The result follows.
\end{proof}


\section{Examples and computations}
\label{sub.compute}


  
\subsection{A non-1-efficient triangulation with two tetrahedra}
\label{sub.nonefficient}

It is traditional in hyperbolic geometry to illustrate theorems 
concerning ideally triangulated manifolds with the case of the standard
triangulation of the complement of the $4_1$ knot. In our examples, we will 
deviate from this principle and begin by giving a detailed computation of the 
state integral for the case of a non-1-efficient ideal triangulation 
with two tetrahedra. This illustrates Propositions 
\ref{prop.ABC}, \ref{prop.quad}, \ref{prop.ptheta}, \ref{prop.balance}
and Theorem~\ref{thm.1}, and also points out the inapplicability of
Theorem~\ref{thm.2}. 

Ideal triangulations can be efficiently described, constructed and
manipulated by \texttt{SnapPy} and \texttt{Regina} \cite{regina, snappy},
and we will follow their description below. In particular, ideal triangulations
can be uniquely reconstructed by their isometry signature, and the latter
is a string of letters and numbers.
There are exactly 10 ideal triangulations with two tetrahedra of manifolds 
with one cusp, given in Table 3 of \cite{Ga:normal}, 
9 of them are 1-efficient, and one of them with isometry signature 
\texttt{cPcbbbdei} is not. Although we will not use it, this triangulation 
is not ordered.
The underlying 3-manifold $M$ is the union of a $T(2,4)$ torus link with the
$T(1,3)$ (trefoil) torus knot.
In \texttt{Regina} format, the tetrahedron gluings of the triangulation
$\calT$ of \texttt{cPcbbbdei} are given by:
\begin{equation*}
\label{regina.gluings3}
\begin{array}{c|c|cccc|}
\text{tet} & \text{glued to} & (012) & (013) & (023) & (123) \\ \hline
0  & &   1 (130) &   1 (023) &   1 (021) &   1 (132) \\
1  & &   0 (032) &   0 (201) &   0 (013) &   0 (132) \\
\end{array}
\end{equation*}
The edges of $\calT$ are given by:
\begin{equation*}
\label{regina.gluings.edges3}
\begin{array}{c|c|cccccc|}
\text{tet} & \text{edge} &  01 & 02 & 03 & 12 & 13 & 23 \\ \hline
0  & &   0 &  0 &  0 &  0 &  1 &  1 \\
1  & &   0 &  0 &  0 &  1 &  0 &  1 \\
\end{array}
\end{equation*}
and the triangle faces of $\calT$ are given by:
\begin{equation*}
\label{regina.gluings.faces3}
\begin{array}{c|c|cccc|}
\text{tet} & \text{face} & 012 & 013 & 023 & 123 \\ \hline
0  & &   0 &  1 &  2 &  3 \\
1  & &   2 &  0 &  1 &  3 \\
\end{array}
\end{equation*}
The gluing equations, in \texttt{SnapPy} format and with the
\texttt{Regina} ordering of the edges, are given by:
\begin{equation*}
\label{gluing.faces3}
\begin{pmatrix}
1 & 1 & 2 & 1 & 2 & 1 \\ 
1 & 1 & 0 & 1 & 0 & 1 \\
0 & -1 & 0 & 1 & 0 & 0 \\ 
0 & 0 & 2 & 0 & 0 & 0 
\end{pmatrix} \,.
\end{equation*}
The $\Abar$, $\Bbar$ and $\Cbar$ matrices are given by:
\begin{equation*}
\Abar = \begin{pmatrix}
1 & 1 \\
1 & 1
\end{pmatrix}, \qquad
\Bbar = \begin{pmatrix}
1 & 2 \\
1 & 0
\end{pmatrix}, \qquad
\Cbar = \begin{pmatrix}
2 & 1 \\
0 & 1
\end{pmatrix} \,.
\end{equation*}
The holonomy of the meridian and the longitude are given by:
$$
\mu=-\b_0 +\a_1, \qquad \lambda=2 \g_0 \,.
$$ 
The $\ve$ variables are given by: 
\be
\label{eq.ex10}
\ve_0 =  \a_0 +\b_0 +2 \g_0 +\a_1+2\b_1+\g_1 -2\pi, 
\qquad 
\ve_1 = \a_0 + \b_0 + \a_1 + \g_1 -2\pi \,.  
\ee
Using $\a_i+\b_i+\g_i=\pi$ for $i=0,1$, we can solve the above equations
in terms of the variables $\a_0,\ve_0,\mu,\lambda$:
\begin{align}
\notag
\a_0 &= \a_0 & \a_1 &= \pi - \a_0 - \lambda/2 + \mu \\ \label{eq.ex10b}
\b_0 &= \pi - \a_0 - \lambda/2 & \b_1 &= \ve_0 - \lambda/2 \\ \notag
\g_0 &= \lambda/2 & \g_1 &= \a_0 - \ve_0 + \lambda - \mu \,.
\end{align}
This illustrates Proposition~\ref{prop.quad}.

If $x_i$ are the variables of the $i$-th edge for $i=0,1$ and $\th$ is a 
pre-angle structure, then 

\begin{align}
\label{eq.ex11}
B(\calT,x,\th)(q)&=c(q)^2 
G_q\left( (-q)^{\frac{\a_0}{\pi}} \frac{x_0 x_0}{x_0 x_1} \right)
G_q\left( (-q)^{\frac{\b_0}{\pi}} \frac{x_0 x_1}{x_0 x_0} \right)
G_q\left( (-q)^{\frac{\g_0}{\pi}} \frac{x_0 x_1}{x_0 x_1} \right)
\\ \notag & \qquad\quad
G_q\left( (-q)^{\frac{\a_1}{\pi}} \frac{x_0 x_1}{x_0 x_0} \right)
G_q\left( (-q)^{\frac{\b_1}{\pi}} \frac{x_0 x_1}{x_0 x_1} \right)
G_q\left( (-q)^{\frac{\g_1}{\pi}} \frac{x_0 x_0}{x_0 x_1} \right)
\\ \notag
&=c(q)^2
G_q\left( (-q)^{\frac{\a_0}{\pi}} \frac{x_0}{x_1} \right)
G_q\left( (-q)^{\frac{\b_0}{\pi}} \frac{x_1}{x_0} \right)
G_q\left( (-q)^{\frac{\g_0}{\pi}} \right)
\\ \notag & \qquad\quad
G_q\left( (-q)^{\frac{\a_1}{\pi}} \frac{x_1}{x_0} \right)
G_q\left( (-q)^{\frac{\b_1}{\pi}}  \right)
G_q\left( (-q)^{\frac{\g_1}{\pi}} \frac{x_0}{x_1} \right) \,.
\end{align}
When $\th$ is positive, the absolutely convergent state-integral 
is given by:

\begin{align}
\label{eq.ex13}
\Ipre_{\calT,\th}(q) &=
\frac{c(q)^2}{(2 \pi \im)^2} \int_{\BT^2}
G_q\left( (-q)^{\frac{\a_0}{\pi}} \frac{x_0}{x_1} \right)
G_q\left( (-q)^{\frac{\b_0}{\pi}} \frac{x_1}{x_0} \right)
G_q\left( (-q)^{\frac{\g_0}{\pi}} \right)
\\ \notag
& \quad\quad
G_q\left( (-q)^{\frac{\a_1}{\pi}} \frac{x_1}{x_0} \right)
G_q\left( (-q)^{\frac{\b_1}{\pi}}  \right)
G_q\left( (-q)^{\frac{\g_1}{\pi}} \frac{x_0}{x_1} \right) 
\frac{dx_0 dx_1}{x_0 x_1} \,.
\end{align}
After rescaling $x_0 \to x_0/x_1$, the integral is free of the $x_1$-variable
and is given by:
\begin{align}
\label{eq.ex14}
\Ipre_{\calT,\th}(q) &=
\frac{c(q)^2}{2 \pi \im} 
G_q\left( (-q)^{\frac{\g_0}{\pi}} \right)
G_q\left( (-q)^{\frac{\b_1}{\pi}}  \right)
\\ \notag
& \times \int_{\BT}
G_q\left( (-q)^{\frac{\a_0}{\pi}} x_0 \right)
G_q\left( (-q)^{\frac{\b_0}{\pi}} \frac{1}{x_0} \right)
G_q\left( (-q)^{\frac{\a_1}{\pi}} \frac{1}{x_0} \right)
G_q\left( (-q)^{\frac{\g_1}{\pi}} x_0 \right) 
\frac{dx_0}{x_0} \,.
\end{align}
Using equation \eqref{eq.ex11}, the above integral becomes:
\begin{align}
\label{eq.ex15}
\Ipre_{\calT,\th}(q) &=
\frac{c(q)^2}{2 \pi \im} 
G_q\left( (-q)^{\frac{\lambda}{2\pi}} \right)
G_q\left( (-q)^{\frac{\ve_0-\lambda/2}{\pi}}  \right)
\int_{\BT}
G_q\left( (-q)^{\frac{\a_0}{\pi}} x_0 \right)
\\ \notag & \quad
G_q\left( (-q)^{1+\frac{-\a_0-\lambda/2}{\pi}} \frac{1}{x_0} \right)
G_q\left( (-q)^{1+\frac{-\a_0-\lambda/2+\mu}{\pi}} \frac{1}{x_0} \right)
G_q\left( (-q)^{\frac{\a_0-\ve_0+\lambda-\mu}{\pi}} x_0 \right) 
\frac{dx_0}{x_0} \,. 
\end{align}
Applying the change of variables 
\be
\label{eq.ex15a}
(-q)^{\frac{\a_0}{\pi}} x_0=y_0 \,,
\ee 
the above integral becomes:
\begin{align}
\label{eq.ex16}
\Ipre_{\calT,\th}(q) &=
\frac{c(q)^2}{2 \pi \im} 
G_q\left( (-q)^{\frac{\lambda}{2\pi}} \right)
G_q\left( (-q)^{\frac{\ve_0-\lambda/2}{\pi}}  \right)
\\ \notag
& \times \int_{C_\a}
G_q\left( y_0 \right)
G_q\left( (-q)^{1-\frac{\lambda/2}{\pi}} \frac{1}{y_0} \right)
G_q\left( (-q)^{1+\frac{-\lambda/2+\mu}{\pi}} \frac{1}{y_0} \right)
G_q\left( (-q)^{\frac{-\ve_0+\lambda-\mu}{\pi}} y_0 \right) 
\frac{dy_0}{y_0} 
\end{align}
where $C_\a$ is the torus $|y_0|=|q^{\frac{\a_0}{\pi}}|$ and the integrand depends
on $\ve_0,\mu,\lambda$ and $y_0$. Notice that when $\th$ is positive, 
Equation~\eqref{eq.ex11} implies that $\ve_0 = \g_0+\b_1 >0$
and $\ve_1 = -\g_0 -\b_1 <0$. Moreover,  we can set $\ve_0=0$
and obtain the uniformly convergent balanced integral
\begin{align}
\label{eq.ex17}
\Ibal_{\calT,\th}(q) &=
\frac{c(q)^2}{2 \pi \im} 
G_q\left( (-q)^{\frac{\lambda}{2\pi}} \right)
G_q\left( (-q)^{\frac{-\lambda/2}{\pi}}  \right)
\\ \notag
& \times \int_{C_\a}
G_q\left( y_0 \right)
G_q\left( (-q)^{1-\frac{\lambda/2}{\pi}} \frac{1}{y_0} \right)
G_q\left( (-q)^{1+\frac{-\lambda/2+\mu}{\pi}} \frac{1}{y_0} \right)
G_q\left( (-q)^{\frac{\lambda-\mu}{\pi}} y_0 \right) 
\frac{dy_0}{y_0} 
\end{align}
which only depends on $(\emu,\elambda)=((-q)^{\mu/\pi},(-q)^{\lambda/(2\pi)})$.
To simplify notation further, let $(s,t)=(\emu,\elambda)$. 
Then, we get:

{\small
\begin{align}
\label{eq.ex18}
I_{\calT,s,t}(q) &=
\frac{c(q)^2}{2 \pi \im} 
G_q\left( t  \right)
G_q\left( t^{-1} \right)
\\ \notag & \times 
\int_{C}
G_q\left( y_0 \right)
G_q\left( (-q) t^{-1} y_0^{-1} \right)
G_q\left( (-q) s t^{-1} y_0^{-1} \right)
G_q\left( s^{-1} t^2 y_0 \right) 
\frac{dy_0}{y_0} 
\end{align}
}

\noindent
where $C$ is the torus $|y_0|=|q|^{\delta_0}$ for small $\delta_0>0$.
This is the meromorphic function of Theorem~\ref{thm.1}.
Next, we compute its singularities, starting from the singularities of the 
integrand. Using part (a) of Lemma~\ref{lem.psi0}, we see that the 
singularities of the integrand are given by
$$
y_0 \in \Sigma^-(s,t) \cup \Sigma^+(s,t) \,,
$$
where
$$
\Sigma^-(s,t) = q^{-\BN} \cup s t^{-2} q^{-\BN}, 
\qquad
\Sigma^+(s,t) = (-q) t^{-1} q^{\BN} \cup (-q) s t^{-1} q^{\BN} \,.
$$
The contour of integration has to separate $\Sigma^-(s,t)$ from 
$\Sigma^+(s,t)$. The integral can only be singular when pinching occurs, 
that is, for $(s,t)$ such that $\Sigma^-(s,t)$ intersects $\Sigma^+(s,t)$. 
This happens precisely when
\be
\label{eq.ScalT}
t \in -q^{\BZ\setminus\{0\}}
\qquad \text{or} \qquad s^{-1} t \in -q^{\BZ\setminus\{0\}} \,. 
\ee
Using the notation of the $q$-rays~\eqref{eq.rse}, the above set
is given by 
$$
-q \Sigma_{0,1} \cup -q^{-1} \Sigma_{0,-1}  \cup
-q \Sigma_{-1,1} \cup -q^{-1} \Sigma_{1,-1} \,.
$$
Thus, the above integral is a meromorphic function which is regular on the
complement of the set $\Sigma_{\calT}$. Note that some of the 
points of the set~\eqref{eq.ScalT} might be regular points of the integral,
this happens for instance when the residue at simple poles vanishes. 
Note also that
the set~\eqref{eq.ScalT} is disjoint from the real torus $|s|=|t|=1$, so the 
integral can be expanded into Laurent series convergent on the real torus 
$|s|=|t|=1$. However, the prefactor $G_q(t) G_q(t^{-1})$ has singularities
on the set $t \in q^{\BZ}$ (and those are actual, i.e., not removable), 
which prevent the meromorphic function
$I_{\calT,s,t}(q)$ from being expanded into Laurent series on the torus
$|s|=|t|=1$. In conclusion, the meromorphic function $I_{\calT,s,t}(q)$
is regular in the complement of the shifted $q$-rays
\be
\label{eq.ex110}
\Sigma_{0,1} \cup \Sigma_{0,-1} \cup
-q \Sigma_{0,1} \cup -q^{-1} \Sigma_{0,-1}  
-q \Sigma_{-1,1} \cup -q^{-1} \Sigma_{1,-1} \,.
\ee

\subsection{The $4_1$ knot}
\label{sub.41}


Next, we discuss the example of the $4_1$ knot, giving the first and
last steps of the above computations, and asking the reader to fill in
the intermediate steps. 

The gluing equations, in \texttt{SnapPy} format and with the
\texttt{Regina} ordering of the edges, are given by:
\begin{equation*}
\begin{pmatrix}
2 & 1 & 0 & 2 & 1 & 0 \\
0 & 1 & 2 & 0 & 1 & 2 \\
1 & 0 & 0 & 0 & 0 & -1 \\
1 & 1 & 1 & 1 & -1 & -3
\end{pmatrix} \,.
\end{equation*}
The above matrix determines the following information.
The $\Abar$, $\Bbar$ and $\Cbar$ matrices are given by:
\begin{equation*}
\Abar = \begin{pmatrix}
2 & 2 \\
0 & 0
\end{pmatrix}, \qquad
\Bbar = \begin{pmatrix}
1 & 1 \\
1 & 1
\end{pmatrix}, \qquad
\Cbar = \begin{pmatrix}
0 & 0 \\
2 & 2
\end{pmatrix} \,.
\end{equation*}
The holonomy of the meridian and the longitude are given by:
$$
\mu= \a_0-\g_1, \qquad \lambda=\a_0 +\b_0+\g_0+\a_1-\b_1-3\g_1 \,.
$$ 
The $\ve$ variables are given by:
$$
\ve_0 = 2\a_0 +\b_0 +2 \a_1 + \b_1  -2\pi, 
\qquad 
\ve_1 = \b_0 + 2 \g_0 + \b_1 +2 \g_1 -2\pi \,.  
$$
Using $\a_i+\b_i+\g_i=\pi$ for $i=0,1$, we can solve the above equations
in terms of the variables $\a_0,\ve_0,\mu,\lambda$:
\begin{align}
\notag
\a_0 &= \a_0 & \a_1 &=  \a_0 + \lambda/2 - \mu \\ \label{eq.ex411}
\b_0 &= \pi - 2 \a_0 + \ve_0 - \lambda/2 & \b_1 
&= \pi - 2 \a_0 - \lambda/2 + 2 \mu
\\ \notag
\g_0 &= \a_0 - \ve_0 + \lambda/2 & \g_1 &= \a_0 - \mu \,.
\end{align}
Let $\th$ be a positive semi-angle structure. Rescaling $x_0 \to x_0/x_1$,
applying the change of variables 
\be
\label{eq.ex412}
(-q)^{-\frac{\a_0}{\pi}} x_0=y_0 \,,
\ee
letting $(s,t)=(\emu,\elambda)$, and following the steps of 
the previous example we obtain that the state-integral is given by
\begin{align}
\label{eq.ex413}
\Ipre_{\calT_{4_1},\th}(q) &=
\frac{c(q)^2}{2 \pi \im} 
\int_{C_\a}
\frac{dy_0}{y_0} 
G_q\left( y_0^{-1} \right)
G_q\left( s^{-1} y_0^{-1} \right)
G_q\left( (-q)^{-\ve_0} t y_0^{-1} \right)
G_q\left( s^{-1} t y_0^{-1} \right)
\\ \notag
& \qquad \qquad \times 
G_q\left( -q (-q)^{\ve_0} t^{-1} y_0^2 \right)
G_q\left( -q s^2 t^{-1} y_0 \right)
\end{align}
where the contour of integration $C_\a$ (determined by~\eqref{eq.ex412}) is
given by $|y_0|=|q^{-\frac{\a_0}{\pi}}|$. Moreover, we can set $\ve_0=0$,  
and obtain the meromorphic function of the $4_1$ knot:

\begin{align}
\label{eq.ex414}
I_{\calT_{4_1},s,t}(q) &=
\frac{c(q)^2}{2 \pi \im} 
\int_{C}
\frac{dy_0}{y_0} 
G_q\left( y_0^{-1} \right)
G_q\left( s^{-1} y_0^{-1} \right)
G_q\left( t y_0^{-1} \right)
G_q\left( s^{-1} t y_0^{-1} \right)
\\ \notag
& \qquad \qquad \times 
G_q\left( -q t^{-1} y_0^2 \right)
G_q\left( -q s^2 t^{-1} y_0 \right) 
\end{align}
where $C$ is the torus $|y_0|=1^+$.

\subsection{The sister of the $4_1$ knot}
\label{sub.41s}


Next, we present the invariant for the sister \texttt{m003} of the $4_1$ knot. 
The gluing equations, in \texttt{SnapPy} format and with the
\texttt{Regina} ordering of the edges, are given by:

\begin{equation*}
\begin{pmatrix}
2 & 0 & 1 & 2 & 0 & 1 \\
0 & 2 & 1 & 0 & 2 & 1 \\
0 & -2 & 0 & 2 & 0 & 0  \\
0 & -1 & 0 & 2 & -1 & 0 
\end{pmatrix} \,.
\end{equation*}

Using the above matrix, we can compute the $\Abar$, $\Bbar$, $\Cbar$ matrices, 
the holonomy of the meridian and longitude, the $\ve$ variables, and express
all variables in terms of the variables $\a_0,\ve_0,\mu,\lambda$:
\begin{align}
\notag
\a_0 &= \a_0 & \a_1 &= \a_0 - \ve_0 + \lambda \\ \label{eq.ex41s1}
\b_0 &= \a_0 - \ve_0 + \lambda - \mu/2 & \b_1 &= \a_0 - \ve_0 + \mu/2
\\ \notag
\g_0 &= \pi - 2 \a_0 + \ve_0 - \lambda + \mu/2 & \g_1 &=  \pi - 2 \a_0 + 2 \ve_0 
- \lambda - \mu/2 \,.
\end{align}
Let $\th$ be a positive semi-angle structure. Rescaling $x_0 \to x_0/x_1$,
applying the change of variables 
\be
\label{eq.ex41s2}
(-q)^{\frac{\a_0}{\pi}} x_0=y_0 \,,
\ee
letting $(s,t)=(\emu,\elambda)$, and following the steps of 
the previous example we obtain that the state-integral is given by
\begin{align}
\label{eq.ex41s3}
\Ipre_{\calT_{m003},\th}(q) &=
\frac{c(q)^2}{2 \pi \im} 
\int_{C_\a}
\frac{dy_0}{y_0} 
G_q\left( (-q)^{1+2\ve_0} s^{-1/2} t^{-2} y_0^{-2} \right)
G_q\left( (-q)^{1+\ve_0} s^{1/2} t^{-2} y_0^{-2} \right)
G_q\left( y_0 \right)
\\ \notag
& \qquad \qquad \times 
G_q\left( (-q)^{-\ve_0} s^{1/2} y_0 \right)
G_q\left( (-q)^{-\ve_0} t^{2} y_0 \right)
G_q\left( (-q)^{-\ve_0} s^{-1/2} t^2 y_0 \right) \,.
\end{align}
where the contour of integration $C_\a$ (determined by~\eqref{eq.ex41s2}) is
given by $|y_0|=|(-q)^{\frac{\a_0}{\pi}}|$. Moreover, we can set $\ve_0=0$,  
and obtain the meromorphic function of the sister of the $4_1$ knot:

\begin{align}
\label{eq.ex41s4}
I_{\calT_{m003},s,t}(q) &=
\frac{c(q)^2}{2 \pi \im} 
\int_{C}
\frac{dy_0}{y_0} 
G_q\left( -q s^{-1/2} t^{-2} y_0^{-2} \right)
G_q\left( -q s^{1/2} t^{-2} y_0^{-2} \right)
G_q\left( y_0 \right)
G_q\left( s^{1/2} y_0 \right)
\\ \notag
& \qquad \qquad \times 
G_q\left( t^{2} y_0 \right)
G_q\left( s^{-1/2} t^2 y_0 \right) \,.
\end{align}
where $C$ is the torus $|y_0|=1^+$. 

\subsection{The unknot}
\label{sub.unknot}

Next, we compute the invariant for the unknot, and find a surprise: we 
can compute the integral exactly. The unknot has three triangulations with
two tetrahedra given in Table 3 of~\cite{Ga:normal}. One of them has 
isometry signature \texttt{cMcabbgds}. 
The gluing equations, in \texttt{SnapPy} format and with the
\texttt{Regina} ordering of the edges, are given by:

\begin{equation*}
\begin{pmatrix}
1 & 2 & 2 & 2 & 2 & 2 \\ 
1 & 0 & 0 & 0 & 0 & 0 \\
0 & 0 & 0 & 0 & 0 & -2 \\ 
1 & 0 & 0 & 0 & -1 & 0
\end{pmatrix} \,.
\end{equation*}
Following the steps of the previous examples, we have:
\be
\label{eq.exunknot0}
\ve_0 = \a_0 + 2\b_0 + 2\g_0 +2 \a_1 + 2\b_1 +2\g_1 -2\pi, \qquad \ve_1 =
\a_0 -2\pi 
\ee
which simplifies to $\ve_0 =-\a_0+2 \pi$ and in particular is bigger than 
$\pi$ when $\th$ is positive. We can express
all variables in terms of the variables $\a_0,\ve_0,\mu,\lambda$:
\begin{align}
\notag
\a_0 &= 2\pi - \ve_0 & \a_1 &= -\pi + \ve_0 + \lambda + \mu/2 
\\ \label{eq.exunknot1}
\b_0 &= \b_0 & \b_1 &= 2\pi - \ve_0 - \lambda \\ \notag
\g_0 &= -\pi - \b_0 + \ve_0 & \g_1 &= -\mu/2 \,.
\end{align}
After rescaling $x_0 \to x_0/x_1$, applying the change of variables 
\be
\label{eq.exunknot2}
(-q)^{\frac{-\b_0}{\pi}} x_0=y_0 \,,
\ee
letting $(s,t)=(\emu,\elambda)$, and following the steps of 
the previous example we obtain that the state-integral is given by
\begin{align}
\label{eq.exunknot3}
\Ipre_{\calT_{\text{unknot}},\th}(q) &=
\frac{c(q)^2}{2 \pi \im} 
G_q\left( (-q)^{2-\ve_0} \right) 
G_q\left( s^{-1/2} \right)
G_q\left( (-q)^{2-\ve_0} t^{-2} \right)
G_q\left( (-q)^{-1+\ve_0} s^{1/2} t^2 \right)
\\ \notag 
& \times 
\int_{C_\a}
\frac{dy_0}{y_0} 
G_q\left( -q^{-1} y_0 \right)
G_q\left( y_0^{-1} \right) \,.
\end{align}
where $C_\a$ is the torus $|y_0|=|(-q)^{-\frac{\b_0}{\pi}}|$.
On the other hand, Equations~\eqref{eq.Gq23} and~\eqref{eq.Gq2} give that
$$
G_q\left( -q^{-1} y_0 \right)
G_q\left( y_0^{-1} \right) = \frac{1}{(1+y_0)(1+q^{-1}y_0)} \,,
$$
so that
\begin{align}
\label{eq.exunknot3a}
\Ibal_{\calT_{\text{unknot}},\th}(q) &=
\frac{c(q)^2}{2 \pi \im} 
G_q\left( (-q)^{2-\ve_0} \right) 
G_q\left( s^{-1/2} \right)
G_q\left( (-q)^{2-\ve_0} t^{-2} \right)
G_q\left( (-q)^{-1+\ve_0} s^{1/2} t^2 \right)
\\ \notag 
& \times 
\int_{C_\a}
\frac{dy_0}{y_0} 
\frac{1}{(1+y_0)(1+q^{-1}y_0)} \,.
\end{align}
When $\th$ is positive, the contour $C_\a$ encircles both singularities 
$-1$ and $-q$ of the integrand, and a residue calculation reveals that 
the integral is zero. It follows that
\be
\label{eq.Iunknot}
I_{\calT_{\text{unknot}},s,t}(q)=0 \,.
\ee

\subsection{The trefoil}
\label{sub.31}

Next, we compute the invariant for the trefoil. 
The gluing equations, in \texttt{SnapPy} format and with the
\texttt{Regina} ordering of the edges, are given by:

\begin{equation*}
\begin{pmatrix}
1 & 0 & 0 & 0 & 1 & 0 \\
1 & 2 & 2 & 2 & 1 & 2 \\
0 & 0 & -1 & 1 & 0 & 0 \\ 
1 & 0 & -4 & 4 & -1 & 0 
\end{pmatrix} \,.
\end{equation*}

The state-integral is given by


{\small
\begin{align}
\label{eq.312}
\Ipre_{\calT_{3_1},\th}(q) &=
\frac{c(q)^2}{2 \pi \im} 
G_q\left( (-q)^{1+\ve_0/2} s^{2} t^{-1} \right)
G_q\left( (-q)^{1+\ve_0/2} s^{-2} t \right) \\ \notag & \times 
\int_{C_\a}
\frac{dy_0}{y_0} 
G_q\left( y_0^{-1} \right)
G_q\left( s^{-1} y_0^{-1} \right)
G_q\left( (-q)^{-\ve_0/2} s^3 t^{-1} y_0 \right)
G_q\left( (-q)^{-\ve_0/2} s^{-2} t y_0 \right) 
\end{align}
}

\noindent
where the contour of integration $C_\a$ is given by 
$|y_0|=|(-q)^{\frac{-\a_1}{\pi}}|$. Moreover, we can set $\ve_0=0$ and 
obtain the meromorphic function given by:


{\small
\begin{align}
\label{eq.313}
I_{\calT_{3_1},s,t}(q) &=
\frac{c(q)^2}{2 \pi \im} 
G_q\left( -q s^{2} t^{-1} \right)
G_q\left( -q s^{-2} t \right) \\ \notag & \times
\int_{C}
\frac{dy_0}{y_0} 
G_q\left( y_0^{-1} \right)
G_q\left( s^{-1} y_0^{-1} \right)
G_q\left( s^3 t^{-1} y_0 \right)
G_q\left( s^{-2} t y_0 \right) 
\end{align}
}
where $C$ is the torus $|y_0|=1^+$.

\subsection{The $5_2$ knot}
\label{sub.52}

Next, we present the invariant of the $5_2$ knot.
The gluing equations, in \texttt{SnapPy} format (we use
the homological longitude and the \texttt{Regina} ordering of the
edges), are given by:

\begin{equation*}
\left(
\begin{array}{ccccccccc}
1 & 1 & 0 & 1 & 0 & 0 & 1 & 1 & 0 \\ 
0 & 1 & 1 & 0 & 0 & 2 & 0 & 1 & 1 \\
1 & 0 & 1 & 1 & 2 & 0 & 1 & 0 & 1 \\
-1 & 0 & 0 & 0 & 0 & 1 & 0 & 0 & 0 \\ 
2 & 0 & -3 & 1 & 0 & -2 & 0 & 0 & 1
\end{array}
\right) \,.
\end{equation*}

After rescaling $x_i \to x_i/x_2$ and a change of variables 
$x_0 = (-q)^{\a_0} y_0$ and $x_1=(-q)^{-\a_1/2} s^{1/2} t^{1/2} y_1$, 
the state-integral is given by

{\small
\begin{align}
\label{eq.522}
\Ipre_{\calT_{5_2},\th}(q) &=
\frac{c(q)^3}{(2 \pi \im)^2} 
\int_{C_\a}
\frac{dy_0 dy_1}{y_0 y_1} 
G_q\left( y_0^{-1} \right)
G_q\left( s y_0^{-1} \right)
G_q\left( (-q)^{-\ve_1} s^2 y_0^{-1} \right)
\\ \notag & \times
G_q\left( -q s^{-2} t^{-1} t^{-1} y_0 y_1^{-2} \right)
G_q\left( (-q)^{1+\ve_0/4} y_0 y_1^{-1} \right)
G_q\left( (-q)^{1+3 \ve_0/4+\ve_1} s^{-3} t^{-1} y_0 y_1^{-1}\right)
\\ \notag
& \times
G_q\left( (-q)^{-\ve_0/4} y_1 \right)
G_q\left( (-q)^{-3\ve_0/4} s t y_1 \right)
G_q\left( s t y_1^{2}  \right) 
\end{align}
}

\noindent
where $C_\a$ is the torus given by 
$|y_0|=|(-q)^{\frac{-\a_0}{\pi}}|$ and $|y_1 s^{1/2} t^{1/2} | = 
|(-q)^{\frac{\a_1}{2\pi}}|$. 
Moreover, we can set $\ve_0=\ve_1=0$ and obtain the 
meromorphic function given by:


{\small
\begin{align}
\label{eq.523}
I_{\calT_{5_2},s,t}(q) &=
\frac{c(q)^3}{(2 \pi \im)^2} 
\int_{C}
\frac{dy_0 dy_1}{y_0 y_1} 
G_q\left( y_0^{-1} \right)
G_q\left( s y_0^{-1} \right)
G_q\left( s^2 y_0^{-1} \right)
G_q\left( -q s^{-2} t^{-1} t^{-1} y_0 y_1^{-2} \right)
\\ \notag & \times
G_q\left( -q y_0 y_1^{-1} \right)
G_q\left( -q s^{-3} t^{-1} y_0 y_1^{-1}\right)
G_q\left(  y_1 \right)
G_q\left(  s t y_1 \right)
G_q\left( s t y_1^{2}  \right) 
\end{align}
}

\noindent
where $C$ is the torus given by 
$|y_0|=1^+$ and $|y_1 s^{1/2} t^{1/2} | = 1^-$. 

\subsection{The $6_1$ knot}
\label{sub.61}

Finally, we present the invariant of the $6_1$ knot.
The gluing equations, in \texttt{SnapPy} format (we use
the homological longitude and the \texttt{Regina} ordering of the
edges), are given by:

\begin{equation*}
\left(
\begin{array}{cccccccccccc}
1 & 0 & 0 & 0 & 0 & 0 & 1 & 1 & 0 & 1 & 1 & 0 \\
0 & 1 & 0 & 1 & 0 & 0 & 1 & 0 & 0 & 0 & 0 & 1 \\
0 & 1 & 1 & 0 & 0 & 2 & 0 & 1 & 1 & 1 & 0 & 0 \\
1 & 0 & 1 & 1 & 2 & 0 & 0 & 0 & 1 & 0 & 1 & 1 \\
-1 & 0 & 0 & 0 & 0 & 1 & 0 & 0 & 0 & 0 & 0 & 0 \\
-1 & 1 & 0 & 0 & 1 & 1 & 0 & -1 & 0 & 0 & 0 & -2 
\end{array}
\right) \,.
\end{equation*}

After rescaling $x_i \to x_i/x_3$ and a change of variables 
\be
\label{eq.611}
x_0 = (-q)^{\a_2} y_0, \qquad x_1=(-q)^{\a_0} y_1, 
\qquad x_2=(-q)^{-\a_1/2} t^{1/2} y_2
\ee
the state-integral is given by


{\small
\begin{align}
\label{eq.612}
\Ipre_{\calT_{6_1},\th}(q) &=
\frac{c(q)^4}{(2 \pi \im)^3} 
\int_{C_\a}
\frac{dy_0 dy_1 dy_2}{y_0 y_1 y_2} 
G_q\left( y_0^{-1} \right)
G_q\left( y_1^{-1} \right)
G_q\left( s y_1^{-1} \right)
G_q\left( (-q)^{\ve_2} s^{-2}y_0^{-1} y_1 \right)
\\ \notag
& \times
G_q\left( -q s^{-1} t^{-1} y_1 y_2^{-2} \right)
G_q\left( (-q)^{1/2-\ve_0/4 + \ve_1/4} t^{-1} y_2^{-1}\right)
G_q\left( (-q)^{3/2 + \ve_0/4 + 3 \ve_1/4} y_0 y_2^{-1} \right)
\\ \notag
& \times
G_q\left( (-q)^{3/2 + 3 \ve_0/4 + \ve_1/4} t^{-1} y_0 y_1 y_2^{-1} \right)
G_q\left( (-q)^{-1/2 - 3 \ve_0/4 - \ve_1/4} t y_1^{-1} y_2 \right)
\\ \notag
& \times
G_q\left( (-q)^{1/2 + \ve_0/4 - \ve_1/4 - \ve_2} s^2 t y_0 y_1^{-1} y_2 \right)
G_q\left( (-q)^{-1/2 - \ve_0/4 - 3 \ve_1/4} y_0^{-1} y_1 y_2 \right)
G_q\left(  -q t y_2^2\right) 
\end{align}
}

\noindent
where $C_\a$ is the torus given by 
$|y_0|=|(-q)^{\frac{-\a_2}{\pi}}|$, $|y_1| = |(-q)^{\frac{-\a_0}{\pi}}|$ and
$|y_2 t^{1/2} | = |(-q)^{\frac{\a_1}{2\pi}}|$.
Moreover, we can set $\ve_0=\ve_1=\ve_2=0$ and obtain the 
meromorphic function given by:

{\small
\begin{align}
\label{eq.613}
I_{\calT_{6_1},s,t}(q) &=
\frac{c(q)^4}{(2 \pi \im)^3} 
\int_{C_\a}
\frac{dy_0 dy_1 dy_2}{y_0 y_1 y_2} 
G_q\left( y_0^{-1} \right)
G_q\left( y_1^{-1} \right)
G_q\left( s y_1^{-1} \right)
G_q\left( s^{-2}y_0^{-1} y_1 \right)
\\ \notag
& \times
G_q\left( -q s^{-1} t^{-1} y_1 y_2^{-2} \right)
G_q\left( (-q)^{1/2} t^{-1} y_2^{-1}\right)
G_q\left( (-q)^{3/2} y_0 y_2^{-1} \right)
\\ \notag
& \times
G_q\left( (-q)^{3/2} t^{-1} y_0 y_1 y_2^{-1} \right)
G_q\left( (-q)^{-1/2} t y_1^{-1} y_2 \right)
\\ \notag
& \times
G_q\left( (-q)^{1/2} s^2 t y_0 y_1^{-1} y_2 \right)
G_q\left( (-q)^{-1/2} y_0^{-1} y_1 y_2 \right)
G_q\left(  -q t y_2^2\right) 
\end{align}
}

\noindent
where $C$ is the torus given by 
$|y_0|=1^+$, $|y_1|=1^+$ and $|y_2 t^{1/2} | = 1^-$. 

\subsection*{Acknowledgements} 
The authors wish to thank R. Siejakowski for a careful reading of the
manuscript and for pointing out typographical errors. 
S.G. was supported in part by grant DMS-14-06419 of the US National Science 
Foundation. R.K. was supported in part by the Swiss National Science 
Foundation and the center of excellence grant “Center for quantum geometry 
of Moduli Spaces” of the Danish National Research Foundation.
The paper was presented by S.G. in a conference at the Max-Planck
Institute in Bonn (Modular Forms are Everywhere) in honor of D. Zagier's
65th birthday. S.G. wishes to thank the organizers, and particularly
K. Bringmann and D. Zagier for their invitation and hospitality. 


\appendix 

\section{A quantum dilogarithm over the LCA group 
$\BT \times \BZ$}
\label{sec.found}
The function~\eqref{eq.psi} with $q$ real has been introduced and studied 
in the functional analytic context of Hilbert spaces and quantum $E(2)$ 
group by Woronowicz in~\cite{Woro:equalities}. In this section, we derive 
some of its operator properties by using the theory of quantum dilogarithms 
over Pontryagin self-dual LCA groups developed 
in~\cite{AK:complex, Kashaev:YB}. 

Throughout the section, for a Hilbert space $H$, we write 
$A\colon H\to H$, if $A$ is a not necessarily bounded linear operator in 
$H$ whose domain is dense in $H$. The Hermitian conjugate of $A$ will be 
denoted $A^*$. Below, we will use freely the standard Dirac's bra-ket 
notation.

The group $\BT \times \BZ$ is a self-dual LCA group with the 
\emph{gaussian exponential}
\begin{equation}
\label{eq:gaus-exp}
\langle \cdot\rangle\colon \BT \times \BZ\to \BT,\quad
\langle z,m\rangle=z^m,\quad \text{for all} \quad (z,m)\in \BT \times \BZ.
\end{equation}
The \emph{Fourier kernel} is fixed as the co-boundary of the gaussian 
exponential
\begin{equation}
\langle z,m;w,n\rangle:= 
\frac{\langle zw,m+n\rangle}{\langle z,m\rangle\langle w,n\rangle}=z^nw^m.
\end{equation}
We define a unitary \emph{Fourier operator} 
$\fourier\colon  L^2(\BT \times \BZ)\to L^2(\BT \times \BZ)$ 
by the integral kernel
\begin{equation}
\langle z,m|\fourier| w,n\rangle=\langle z,m;w,n\rangle.
\end{equation}

For any (measurable) function $f\colon \BT \times \BZ\to \BC$, we associate 
three normal operators as follows. The multiplication operator by $f$:
\begin{equation}
f(\mypos)\colon  L^2(\BT \times \BZ)\to L^2(\BT \times \BZ),
\quad \langle z,m|f(\mypos)=f(z,m)\langle z,m|,
\end{equation}
its unitary conjugate by the Fourier operator
\begin{equation}
f(\mymom):=\fourier f(\mypos)\fourier^*,
\end{equation}
and the unitary conjugate of the latter by the inverse of the (unitary) 
multiplication operator $\langle\mypos\rangle$ by the gaussian
exponential~\eqref{eq:gaus-exp}:
\begin{equation}
f(\mymom+\mypos):=\langle\mypos\rangle^*f(\mymom)\langle\mypos\rangle.
\end{equation}
We remark that all three operators $f(\mypos)$, $f(\mymom)$ and 
$f(\mymom+\mypos)$ have spectrum given by the closure of the image of $f$.

\begin{lemma}
\label{lemma-fun-eta}
The function
\begin{equation}
\myfun\colon \BT \times \BZ\to \BC_{\ne0},\quad (z,m)\mapsto-zq^{1-m},
\end{equation}
satisfies the following operator equations:
\begin{equation}
\label{eq:eta-rel}
\myfun(\mypos)\myfun(\mymom)=q^2\myfun(\mymom)\myfun(\mypos),
\quad\myfun(\mypos)^*\myfun(\mymom)=\myfun(\mymom)\myfun(\mypos)^*,\quad 
\myfun(\mymom+\mypos)=-\myfun(\mymom)\myfun(\mypos).
\end{equation}
\end{lemma}

\begin{proof}
In the Hilbert space $L^2(\BZ)$, define a self-adjoint operator $\myh$ by
\begin{equation}
\langle m|\myh=m\langle m|,\quad \text{for all} \quad m\in \BZ,
\end{equation}
and a unitary operator $\myz$ by
\begin{equation}
\langle m| \myz=\langle m-1|,\quad \text{for all} \quad m\in \BZ.
\end{equation}
These operators satisfy the commutation relation
\begin{equation}
[\myh,\myz]:=\myh\myz-\myz\myh=\myz
\end{equation}
which is verified as follows:
\begin{equation}
\langle m|[\myh,\myz]=\langle m|\myh\myz-\langle m|
\myz\myh=m\langle m-1|-(m-1)\langle m-1|=\langle m|\myz.
\end{equation}
Similarly, in the Hilbert space $L^2(\BT)$, we define a self-adjoint operator 
$\myH$ by
\begin{equation}
\langle z|\myH=z\frac{\partial}{\partial z}\langle z|,\quad \text{for all} 
\quad z\in \BT,
\end{equation}
and a unitary operator $\myZ$ by
\begin{equation}
\langle z| \myZ=z\langle z|,\quad \text{for all} \quad z\in \BT.
\end{equation}
It is easily verified that
\begin{equation}
[\myH,\myZ]:=\myH\myZ-\myZ\myH=\myZ.
\end{equation}
Moreover, if $\myi\colon L^2(\BZ)\to L^2(\BT)$ is the isomorphism defined 
by the integral kernel
\begin{equation}
\langle z|\myi|m\rangle=z^m,\quad \text{for all} \quad (z,m)\in \BT\times\BZ,
\end{equation}
then we have the equalities
\begin{equation}
\myi \myh=\myH\myi,\quad \myi \myz=\myZ\myi.
\end{equation}
Identifying $\myH,\myZ,\myh,\myz$  with their natural counterparts in 
$L^2(\BT\times\BZ)$, we have
\begin{equation}
\fourier\myh=\myH\fourier,\quad\fourier\myz=\myZ\fourier,
\quad\fourier\myH=-\myh\fourier,\quad\fourier\myZ=\myz^{-1}\fourier,
\end{equation}
and
\begin{equation}
\myh\langle\mypos\rangle=\langle\mypos\rangle\myh,
\quad \myZ\langle\mypos\rangle=\langle\mypos\rangle\myZ,\quad
\myH\langle\mypos\rangle=\langle\mypos\rangle(\myH+\myh),
\quad \myz\langle\mypos\rangle=\langle\mypos\rangle\myz\myZ^{-1}
\end{equation}
so that
\begin{equation}
f(\mypos)=f(\myZ,\myh),\quad f(\mymom)=f(\myz^{-1},\myH),
\quad f(\mymom+\mypos)=f(\myz^{-1}\myZ,\myH+\myh)
\end{equation}
for any $f\colon  \BT \times \BZ\to \BC$, where, in the right hand sides, 
the functions with operator arguments are understood in the spectral sense. 
In the case of $f=\eta$, we thus have
\begin{equation}
\myfun(\mypos)=-\myZ q^{1-\myh},\quad \myfun(\mymom)=-\myz^{-1} q^{1-\myH},
\quad \myfun(\mymom+\mypos)=-\myz^{-1} \myZ q^{1-\myH-\myh}
\end{equation}
and the relations~\eqref{eq:eta-rel} are verified straightforwardly.
\end{proof}
Next, observe that the function
\begin{equation}
\mymu\colon  \BT \times \BZ\to \BC,\quad 
(z,m)\mapsto \poc{\myfun(z,m);q^2}{\infty},
\end{equation}
nowhere vanishes and satisfies the operator five term identity
\begin{equation}
\label{eq:oper5pq}
\mymu(\mymom)\mymu(\mypos)=\mymu(\mypos)\mymu(\mymom+\mypos)\mymu(\mymom)
\end{equation}
as a consequence of Lemma~\ref{lemma-fun-eta} and the formal power series 
identity in non-commuting indeterminates
\begin{equation}
\label{eq:oper5}
\poc{\myv;q^2}{\infty}\poc{\myu;q^2}{\infty}
=\poc{\myu;q^2}{\infty}\poc{-\myv\myu;q^2}{\infty}\poc{\myv;q^2}{\infty},
\quad \myu\myv=q^2\myv\myu,
\end{equation}
which is equivalent to the $q$-binomial formula
\begin{equation}
\sum_{n=0}^\infty\frac{\poc{a;q^2}{n}}{\poc{q^2;q^2}{n}}z^n
=\frac{\poc{az;q^2}{\infty}}{\poc{z;q^2}{\infty}},
\end{equation}
see~\cite{Kashaev2004}. Indeed, by substituting $\myu$ by $\eta(\mypos)$ 
and $\myv$ by $\eta(\mymom)$, we convert \eqref{eq:oper5} into 
\eqref{eq:oper5pq}. Lemma~\ref{lemma-fun-eta} also implies that the function
\begin{equation}
\label{eq:qdlTZ}
\phi_q\colon  \BT \times \BZ\to \BT,\quad 
(z,m)\mapsto \mymu(z,m)/\mymu(1/z,m)=\mymu(z,m)/\overline{\mymu(z,m)},
\end{equation}
satisfies the unitarized version of identity~\eqref{eq:oper5pq}:
\begin{equation}
\label{eq:oper5pqunitary}
\phi_q(\mymom)\phi_q(\mypos)=\phi_q(\mypos)\phi_q(\mymom+\mypos)\phi_q(\mymom).
\end{equation}
Moreover, $\phi_q$ satisfies an  inversion relation, see below 
\eqref{eq:phi-ir}, which allows us to identify it as an example of a 
quantum dilogarithm over $\BT \times \BZ$.

\begin{lemma}
The function \eqref{eq:qdlTZ} satisfies the following inversion relation
\begin{equation}
\label{eq:phi-ir}
\phi_q(z,m)\phi_q(1/z,-m)=z^m=\langle z,m\rangle,\quad 
\text{for all} \quad (z,m)\in \BT \times \BZ.
\end{equation}
\end{lemma}

\begin{proof}
We have
\begin{multline}
\phi_q(z,m)\phi_q(1/z,-m) =\frac{\poc{-q^{1-m}z;q^2}{\infty}
\poc{-q^{1+m}/z;q^2}{\infty}}{\poc{-q^{1-m}/z;q^2}{\infty}
\poc{-q^{1+m}z;q^2}{\infty}}\\
=\frac{\theta_{q}(z q^{-m})}{\theta_{q}(z q^{m})}=
\frac{\theta_{q}(q^{-2m}z q^{m})}{\theta_{q}(z q^{m})}
=q^{-m^2}\left(zq^m\right)^m=z^m
\end{multline}
where in the second equality we have used \eqref{eq.theta_q} and in the 
forth equality the functional equation~\eqref{eq:fe-theta}.
\end{proof}


\section{The quantum dilogarithm and the Beta pentagon relation}

In this section, we generalize the result of \cite{AK:complex} to include 
the LCA groups which do not admit division by $2$. Note that most of the
equations of this section (for instance, \eqref{eq:faddeev-type-int}, 
\eqref{eq:faddeev-type-intc}, \eqref{eq:int-id}, \eqref{eq:int-b-id}) 
are valid as distributions. 

Let $A$ be a self-dual LCA group with a gaussian exponential 
$\langle\cdot\rangle \colon A\to \BT$ and the Fourier kernel
\begin{equation}
\langle x;y\rangle=
\frac{\langle x+y\rangle}{\langle x\rangle\langle y\rangle},
\quad \text{for all} \quad (x,y)\in A^2.
\end{equation}
For a non-negative integer $n \in \BN$, denote
\begin{equation}
[n]:=\BZ_{\ge0}\cap\BZ_{\le n},\quad \text{for all} \quad n\in\BZ_{\ge0} \,.
\end{equation}
According to~\cite{AK:complex}, a bounded function
\begin{equation}
f\colon [4]\times A\to \BC,\quad (i,x)\mapsto f_i(x),
\end{equation}
is called of \emph{Faddeev type} if it satisfies the non-constant version 
of the operator pentagon relation
\begin{equation}
f_1(\mymom)f_3(\mypos)=f_4(\mypos)f_2(\mymom+\mypos)f_0(\mymom) \,.
\end{equation}
The latter is equivalent to the functional integral identity
\begin{equation}
\label{eq:faddeev-type-int}
\tilde f_1(x)\tilde f_3(u) \langle x;u\rangle =
\int_A \tilde f_4(u-z)\tilde f_2(z)\tilde f_0(x-z)\langle z\rangle 
\operatorname{d}\!z,\quad \text{for all} \quad (x,u)\in A^2,
\end{equation}
where
\begin{equation}
\tilde f_i(x):=
(\fourier^{-1} f_i)(x)=\int_A \langle x;-y\rangle f_i(y)\operatorname{d}\!y.
\end{equation}
Defining
\begin{equation}
\hat f_i(x):= \tilde f_i(x)\langle x\rangle,
\end{equation}
we rewrite \eqref{eq:faddeev-type-int} in the form
\begin{equation}
\label{eq:faddeev-type-intc}
\hat f_1(x)\hat f_3(u) =
\int_A \langle x-z;z-u\rangle \hat f_4(u-z)\hat f_2(z)
\hat f_0(x-z)\operatorname{d}\!z,\quad\text{for all} \quad (x,u)\in A^2,
\end{equation}
We call a subgroup $B\subset A$ \emph{isotropic} if it satisfies the condition
\begin{equation}
\langle b;b'\rangle=1, \quad \text{for all} \quad (b,b')\in B^2.
\end{equation}
\begin{lemma}
Let a bi-character $\chi\colon A^2\to \BT$ be such that 
\begin{equation}
\langle x;y\rangle=\chi(x,y)\chi(y,x),\quad \text{for all} \quad (x,y)\in A^2,
\end{equation}
and $B\subset A$ an isotropic subgroup. Then, for any  function $f_i(x)$ 
of Faddeev type, the function
\begin{equation}
g\colon [4]\times A^2\to\BC,\quad (i,x,y)\mapsto 
g_i(x,y):=\chi(x,y)\int_B\hat f_i(x+b)\langle b;y\rangle\operatorname{d}\!b
\end{equation}
is automorphic, i.e., satisfies
\begin{multline}
\label{eq:autom}
g_i(x+b,y)=\chi(-y,b)g_i(x,y),\quad g_i(x,y+b)=\chi(x,b)g_i(x,y),\\ 
\text{for all} \quad (i,b,x,y)\in [4]\times B\times A^2,
\end{multline}
and satisfies the integral identity
\begin{equation}
\label{eq:int-id}
g_1(x,y)g_3(u,v)=
\int_{A/B}g_4(u-z,v+z-x) g_2(z,y+v)g_0(x-z,y+z-u)\operatorname{d}\! z
\end{equation}

\begin{proof} 
The automorphicity properties~\eqref{eq:autom} are verified in a 
straightforward manner, 
while to derive the integral identity~\eqref{eq:int-id}, we write 
\begin{multline}
\frac{g_1(x,y)g_3(u,v)}{\chi (x,y)\chi(u,v)}\\
=\int_{A\times B^2} \langle x+b-z;z-u-c\rangle\langle b;y\rangle
\langle c;v\rangle \hat f_4(u+c-z)\hat f_2(z)\hat f_0(x+b-z)
\operatorname{d}(z,b,c)\\
=\int_{A\times B^2} \langle x-z;z-u\rangle\langle b;y+z-u\rangle
\langle c;v+z-x\rangle \hat f_4(u+c-z)\hat f_2(z)\hat f_0(x+b-z)
\operatorname{d}(z,b,c)\\
=\int_{A}\frac{ \langle x-z;z-u\rangle g_4(u-z,v+z-x)
\hat f_2(z)g_0(x-z,y+z-u)}{ \chi(u-z,v+z-x)\chi(x-z,y+z-u)}
\operatorname{d}\!z\\
=\int_{A}\frac{g_4(u-z,v+z-x)\hat f_2(z)g_0(x-z,y+z-u)}{ 
\chi(u-z,v)\chi(x-z,y)}\operatorname{d}\!z
\end{multline}
so that
\begin{multline}
 g_1(x,y)g_3(u,v)=\int_{A}\chi(z,y+v)g_4(u-z,v+z-x)\hat f_2(z)g_0(x-z,y+z-u)
\operatorname{d}\!z\\
=\int_{(A/B)\times B}\chi(z+b,y+v)g_4(u-z-b,v+z+b-x)\hat f_2(z+b)
g_0(x-z-b,y+z+b-u)\operatorname{d}(z,b)\\
=\int_{(A/B)\times B}\chi(z+b,y+v)\chi(v+y,b)g_4(u-z,v+z-x)
\hat f_2(z+b)g_0(x-z,y+z-u)\operatorname{d}(z,b)\\
=\int_{A/B}g_4(u-z,v+z-x) g_2(z,y+v)g_0(x-z,y+z-u)\operatorname{d}\! z
\end{multline}
\end{proof}

\noindent
Now, we point out that the integral identity~\eqref{eq:int-id} is an equivalent 
form of the automorphic Beta pentagon identity. Namely, if we define
\begin{equation}
\phi_i(x,y):=g_i(-y,x+y)\quad \Leftrightarrow\quad g_i(x,y)=\phi_i(x+y,-x),
\end{equation}
then
\begin{equation}
\label{eq:int-b-id}
\phi_1(x,y)\phi_3(u,v)=\int_{A/B}\phi_4(u+y,v-z) \phi_2(x+y+u+v-z,z)
\phi_0(x+v,y-z)\operatorname{d}\! z \,.
\end{equation}
\end{lemma}
In the case of the quantum dilogarithm $\phi_q$ constructed in 
Appendix~\ref{sec.found}, we have 
\begin{equation}
A=\BT\times\BZ,\quad \langle z,m\rangle=z^m,\quad\langle z,m;w,n\rangle=z^nw^m
\end{equation}
so that
\begin{align*}
\hat\phi_q(z,m) &=z^m \int_{\BT\times\BZ}\phi_q(t,k)\langle z,m;1/t,-k\rangle 
\operatorname{d}(t,k)
=z^m \int_{\BT\times\BZ}\phi_q(t,k)z^{-k}t^{-m}\operatorname{d}(t,k)\\
& =z^m \int_{\BT\times\BZ}\psi(1/t,k)z^{-k}t^{-m}\operatorname{d}(t,k)
=z^m \int_{\BT}\psi^0(1/t,1/z)t^{-m}\frac{\operatorname{d}\!t}{2\pi i t} \,.
\end{align*}
We choose
\begin{equation}
B=\BZ\subset\BT\times\BZ
\end{equation}
and
\begin{equation}
\chi((z,m),(w,n))=w^m,
\end{equation}
so that the automorphic factors are trivial:
\begin{equation}
\chi((z,m),(1,n))=1,\quad \text{for all} \quad (z,m,n)\in\BT\times\BZ^2.
\end{equation}
Thus,
\begin{align*}
g((z,m),(w,n))& =\chi((z,m),(w,n))\sum_{k\in\BZ}\hat\phi_q(z,m+k)
\langle 1,k;w,n\rangle\\ &
=w^m\sum_{k\in\BZ}\hat\phi_q(z,m+k)w^k=\sum_{k\in\BZ}\hat\phi_q(z,k)w^k \\ 
& = 
\int_{\BT}\sum_{k\in\BZ}(zw/t)^k\psi^0(1/t,1/z) 
\frac{\operatorname{d}\!t}{2\pi i t} \\ &
= \int_{\BT}\delta_{\BT}(zw/t)\psi^0(1/t,1/z)
\frac{\operatorname{d}\!t}{2\pi i t} =\psi^0(1/(zw),1/z)
\end{align*}
and 
\begin{equation}
\phi(((z,m),(w,n))=g((1/w,-n),(zw,m+n))=\psi^0(1/z,w).
\end{equation}
Taking into account the symmetry of the Beta pentagon identity under the 
negation of all arguments, we conclude that $\psi^0(1/z,w)$ and 
$\psi^0(z,1/w)$ both satisfy the Beta pentagon identity.


\bibliographystyle{hamsalpha}
\bibliography{biblio}
\end{document}